\documentclass{EngC}
      \usepackage{soul}        
     \usepackage{rotating}
  \usepackage{floatpag}
  \rotfloatpagestyle{empty}

\usepackage{amsmath,amsfonts,empheq,latexsym,epsfig,subfigure,fancybox,verbatim,paralist,mdframed,framed,amsthm,algorithm,isomath,multirow,makeidx,cancel,enumitem,algorithmic,minitoc,url,varioref,datetime,setspace,array}

\usepackage{vikram}

\usepackage{mdframed}

\usepackage[framed]{mcode}

\usepackage{fancyheadings,etoolbox}
\usepackage[title]{appendix}

\usepackage{pstricks-add,pst-sigsys,pst-tree,pst-bezier,pst-plot}
\usepackage[stable]{footmisc}

\AtBeginEnvironment{subappendices}{%
  \addtocontents{toc}{\protect\addvspace{10pt}{\em Appendix to Chapter \thechapter}}
}

\renewcommand{\thepart}{Part \Roman{part}.}
\usepackage{titlesec}
\titleformat{\part}{\huge\bfseries}{}{0pt}{\begin{center}\thepart\end{center}}

\newcounter{rowno}
\allowdisplaybreaks

\makeindex

\newcommand{\secn}{\S}
\newcommand{\chp}{\S}

\makeatletter
\def\nl#1#2{\begingroup
    #2%
    \def\@currentlabel{#2}%
    \phantomsection\label{#1}\endgroup
}
\makeatother

{\ensuremath{
\begin{empheq}[box=\fbox]{align}
{#1}
\end{empheq}
}}

\newtheorem{theorem}            {Theorem}[section]
\newtheorem{corollary}          [theorem]{Corollary}

\newtheorem{definition}         [theorem]{Definition}

\newtheorem{lemma}              [theorem]{Lemma}

\newcommand{\normal}{\mathbfit{N}}  

{
\begin{mdframed}
\par\noindent\textbf{#1:}\begin{rmfamily}\noindent}%
{\end{rmfamily}
\end{mdframed}
} 

\newcommand{\pwe}{Complements and Sources}

\newcommand{\I}{\Pi}
\newcommand{\avg}{\phi}

\newcommand{\dob}{\rho}  

\newcommand{\bias}{\operatorname{Bias}}
\newcommand{\var}{\operatorname{Var}}
\newcommand{\msd}{\operatorname{MSD}}

\newcommand{\pdf}{p}
\newcommand{\cdf}{F}
\newcommand{\prob}{\mathbb{P}}
\newcommand{\E}                 {\Bbb{E}}

\renewcommand{\P}                 {\Bbb{P}}
\newcommand{\cov}{\operatorname{cov}}

\newcommand{\obs}{y}

\newcommand{\onoise}{v}

\newcommand{\onoisecovnew}{R}

\newcommand{\state}{x}
\newcommand{\statespace}{\mathcal{X}}

\newcommand{\statedim}{X}

\newcommand{\dvar}[2]{\|{#1}-{#2}\|_{\text{\tiny{TV}}}}

\newcommand{\oprob}{B}
\newcommand{\tp}{P}

\newcommand{\finaltime}{N}

\newcommand{\model}{\theta}

\newcommand{\truemodel}{\model^o}

\newcommand{\Model}{\Theta}

\newcommand{\beliefm}{\belief_\model}

\newcommand{\tpm}{{\tp_\model}}
\newcommand{\nablam}{\nabla_\model}

\newcommand{\belief}{\pi}

\newcommand{\bbelief}{\bar{\pi}}

\newcommand{\kalmancov}{\Sigma}

\newcommand{\ole}{\stackrel{\text{defn}}{=}}

\newcommand{\argmin}{\operatornamewithlimits{argmin}}
\newcommand{\argmax}{\operatornamewithlimits{argmax}}

\newcommand{\reals}{{\rm I\hspace{-.07cm}R}}

\newcommand{\beq}{\begin{equation}}
\newcommand{\eeq}{\end{equation}}
\newcommand{\nn}{\nonumber}

\renewcommand{\(}		{\left(}
\renewcommand{\)}		{\right)}

\renewcommand{\th}{\theta}
\newcommand{\Th}{\Theta}

\newcommand{\p}{\prime}

\newcommand{\one}{\mathbf{1}}

\newcommand{\zero}{\mathbf{0}}

\newcommand{\f}{f} 

\newcommand{\diag}{\textnormal{diag}}

\newcommand{\hQ}{\hat{Q}}

\newcommand{\cost}{c}

\newcommand{\Cost}{C}

\newcommand{\action}{u}
\newcommand{\baction}{\bar{\action}}
\newcommand{\actionspace}{\,\mathcal{U}}
\newcommand{\actiondim}{U}

\newcommand{\discount}{\rho}

\def \con {{\beta}}
\def \rcon {{\gamma}}

\def\param{{\alpha}}

\def \a {{\psi}}
\def\Param{{\boldsymbol{\alpha}}}

 \def \G{{B}}
 
 \newcommand{\Ep}{\E_{\policy}}

\newcommand{\policy}{\mu}

\newcommand{\optpolicy}{\policy^*}

\newcommand{\valuef}{V}

\newcommand{\info}{\mathcal{I}}

\newcommand{\thstate}{\state}
\newcommand{\thstatedim}{\statedim}

\def \La {{\cal L}}
\def \GF{{\widehat{\nabla C}^{\text{WD}}}}
\def \GS{{\widehat{\nabla C}^{\text{Score}}}}

\newcommand{\wderiv}{g}

\newcommand{\stepsize}{\epsilon}

\newcommand{\bm}{w}
\newcommand{\globalopt}{\mathcal{G}}
\newcommand{\cd}{(\cdot)}
\newcommand{\br}{{b}^{\gamma}}
\newcommand{\bbr}{b^{\gamma}}
\newcommand{\boldf}{f}

\def\lb{\left[}
\def\rb{\right]}
\def\lbr{\left\lbrace}
\def\rbr{\right\rbrace}

\newcommand{\maxdiff}{D}

\newcommand{\score}{S}
\newcommand{\batchsize}{N}

\newcommand{\degreelink}{\th}

\newcommand{\penaltyweight}{\Delta}

\def\ph{\varphi}
\newcommand{\M}{{\cal M}}
\newcommand{\lbar}{\overline}
\newcommand{\wdt}{\widetilde}

\newcommand{\ad}{&\!\!\!\disp}
\newcommand{\aad}{&\disp}
\newcommand{\barray}{\begin{array}{ll}}
\newcommand{\earray}{\end{array}}
\newcommand{\disp}{\displaystyle}

\def\cd{(\cdot)}
\newcommand{\e}{\varepsilon}
\newcommand{\tth}{\tilde{\th}}
\newcommand{\numagents}{M}

\newcommand{\nature}{\state}
\newcommand{\pop}{\theta}
\newcommand{\bpop}{\bar{\pop}}

\newcommand{\popspace}{\Theta}

\newcommand{\degdist}{\rho}

\newcommand{\wdh}{\widehat}

\newcommand{\mtg}{\nu}

\newcommand{\simplex}{\Pi}

\newcommand{\batch}{\iota}
\newcommand{\bsize}{N}
\newcommand{\bindex}{n}

\begin{document}

\title[Stochastic Approximation Algorithms for Markov Decision Processes]{Reinforcement Learning}
\author{Vikram Krishnamurthy\\ University of British Columbia, \\ Vancouver, Canada. V6T   1Z4.\\  {\tt vikramk@ece.ubc.ca} \\ December 2015}

\maketitle

\dominitoc


\frontmatter

\mainmatter
\chapter{Introduction}
Stochastic adaptive control algorithms are of  two types: {\em direct
  methods}, where the unknown transition probabilities $\tp_{ij}(\action)$ 
are estimated
 simultaneously while updating the control policy, and {\em
  implicit methods}  (such as simulation based methods), where the
transition probabilities are not directly estimated in order to
compute the control policy. Reinforcement learning algorithms are implicit methods for adaptive control.
 The aim of this article is to present elementary results
in stochastic approximation algorithms for reinforcement learning of Markov decision processes. 
This article is a short summary of a much more detailed presentation in the forthcoming book \cite{Kri16}.

Recall that:
\begin{compactitem}
\item A Markov Decision Process (MDP)  is obtained by controlling the transition probabilities of a Markov chain as
it evolves over time.
\item
A Hidden Markov Model (HMM) is a noisily observed Markov chain.
\item A partially observed Markov decision process (POMDP) is  obtained by controlling
the transition probabilities and/or observation probabilities of an HMM. 
\end{compactitem}
 These relationships are illustrated in Figure \ref{fig:terminology}.

A POMDP specializes to a MDP if
the observations are  noiseless and equal to the state of the Markov chain.
A POMDP specializes to an HMM  if the control is removed.
Finally,  an HMM specializes to a Markov chain if the observations are  noiseless and equal to the state of the Markov chain.

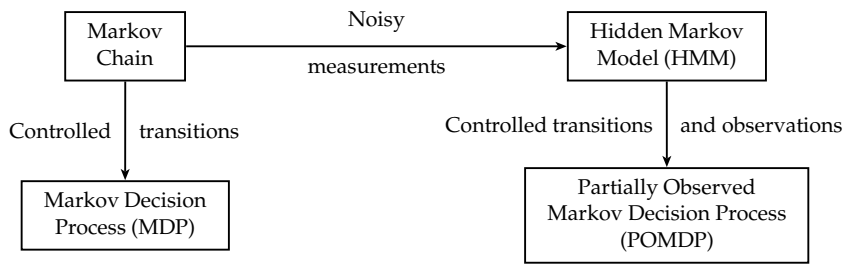
\begin{figure}[h] \centering
\scalebox{0.9}{
\begin{pspicture}[showgrid=false](0.5,-2.2)(14,1.2)


\dotnode[dotstyle=square*,dotscale=0.001](15.5,0.75){dot}

\psblock(2,.75){H1}{\begin{tabular}{c} Markov  \\ Chain \end{tabular}}

\psblock(10,.75){B2}{\begin{tabular}{c}  Hidden Markov  \\ Model (HMM) \end{tabular}}

\psblock(2,-1.75){B1}{\begin{tabular}{c} Markov Decision \\  Process  (MDP)\end{tabular}}

\psblock(10,-1.75){B3}{\begin{tabular}{c} Partially Observed \\  Markov Decision Process \\ (POMDP) \end{tabular}}




\psset{style=Arrow}

\ncline{H1}{B2} \naput[npos=.5]{Noisy}
\nbput[npos=.5]{measurements}

\ncline{H1}{B1} \nbput[npos=.5]{Controlled}
\naput[npos=.5]{transitions}

\ncline{B2}{B3} \nbput[npos=.5]{Controlled transitions}
\naput[npos=.5]{and observations}

\end{pspicture}}
\caption{Terminology of HMMs, MDPs and POMDPs} \label{fig:terminology}
\end{figure}

This article deals  with {\em stochastic gradient algorithms} for
estimating reasonable (locally optimal) strategies for MDPs and POMDPs. 

Suppose a decision maker can observe the noisy response $\obs_k$ of a controlled stochastic system to any action $\action_k$ that it chooses.
Let $\info_k = \{\action_0,\obs_1,\ldots,\action_{k-1},\obs_k\}$ denote the history of actions
and observed responses up to  time $k$.
The decision
maker chooses its action as $\action_ k = \policy_\model(\info_k)$ where 
 $\policy_\model$ denote a parametrized policy (parametrized by a vector $\model$). 
Then to optimize its choice of actions,
the decision maker  needs to compute  the optimal parameter $\model^*$ which minimizes the cost criterion $\E\{\Cost(\model, \info_k)\}$.
The decision maker uses the
 following stochastic gradient algorithm to estimate  $\model^*$:
\begin{empheq}[box=\fbox]{equation} \model_{k+1} = \model_k - \stepsize\,  \nablam \Cost(\model_k,\info_k), \quad k=0,1,\ldots .\label{eq:introsa} 
\end{empheq}
Here $\nablam \Cost(\model_k,\info_k)$ denotes the gradient (or estimate of gradient) of the instantaneous  cost  with respect to the parameter $\th$ and $\epsilon$ denotes
a small positive step size. Such algorithms lie within the class of reinforcement learning methods since the past experience $\info_k$ is used
to adapt the parameter $\model_k$ which in turn determines the actions; intuitively a good choice of $\model$ would result in good performance which
in turn reinforces this choice.
This article  deals with such stochastic gradient algorithms including  how to compute the gradient estimate and  analyze the resulting algorithm.


Chapter \ref{chp:simulation} discusses gradient estimation for Markov processes
via stochastic simulation. This forms the basis of gradient based
reinforcement learning.

Chapter \ref{chp:rl} presents simulation based stochastic approximation  algorithms for estimating the optimal policy of MDPs when the transition probabilities are not known. These algorithms are also described in the context of  POMDPs. The Q-learning algorithm and 
gradient based reinforcement learning algorithms are presented.

Chapter \ref{chp:markovtracksa} gives a brief description of convergence analysis of stochastic approximation algorithms. Examples
given include recursive maximum likelihood estimation of HMM parameters,  the least mean squares algorithm for estimating the state of an HMM
(which can be used for adaptive control of a POMDP), discrete stochastic optimization algorithms, and mean field dynamics for approximating the dynamics of information
flow in large scale social networks.

 For further details, illustrative examples,  proofs of theorem and more insight, please see \cite{Kri16}. More specifically, \cite{Kri16} deals extensively with POMDPs and
 structural results. The algorithms in this article can be used in conjunction with the structural results resulting in efficient numerical implementations. Structural results
 for POMDPs are developed in \cite{KD07,KW09,KD09,Kri11,Kri12,Kri13}.

\chapter{Stochastic Optimization and Gradient Estimation} \label{chp:simulation}

\minitoc 

Consider the discounted cost MDP problem:
$$\text{ Compute } \optpolicy = \argmin_\policy \E_{\policy} \bigl\{ \sum_{k=0}^\infty \discount^k \cost(\state_k, \policy(\state_k))\bigr\}.
$$
where $x_k$ is the controlled state of the system, $\cost(x_k,u_k)$ is the cost incurred at time $k$ by choosing action
$u_k = \policy(\state_k)$ and $\rho < 1$ is a discount factor.
 Stochastic dynamic programming is used to compute the {\em globally} optimal policy $\optpolicy$ for such problems.

In comparison,  this article   deals with computing (estimating)   {\em local} minima using stochastic gradient algorithms.
Suppose the action at time $k$  is  chosen according to the following  parameterized  policy:
$\action_k = \policy_\model(\state)$ for some pre-specified function $\policy_\model$ 
parametrized by the vector  $\model \in \reals^p$.
Then the  aim is:
\beq  \text{ Compute }  \th^* = \argmin_{\model } C(\model), \quad C(\model) = \E_{\policy_\model} \bigl\{ \sum_{k=0}^\infty \discount^k
 \cost(\state_k, \policy_\model(\state_k)\bigr\}. \label{eq:localmin} \eeq
 This will be achieved  using a stochastic gradient algorithm of the form
 \beq  \model_{k+1} = \model_k - \stepsize_k \widehat\nabla_\model \Cost_k(\model_k) . \label{eq:sgrad0} \eeq
Here $\stepsize_k$ is a scalar step size and  $\widehat \nabla_\model \Cost_k(\model_k)$ denotes an   estimate
of gradient ${\nabla}_\model \Cost(\model)$ evaluated at $\model_k$.  These gradient estimates 
need to be computed using the observed realization  of $\{\state_k, \{c(\state_k,\action_k)\}$
 since this is the only information available to the decision  maker.
Since in general, $C(\model)$ is non-convex, at best one can expect (\ref{eq:sgrad0}) to converge (in a sense
to be made precise below)  to a local stationary point of (\ref{eq:localmin}).
 
Even though stochastic gradient algorithms  typically converge to a local stationary point, there are several advantages 
compared to stochastic dynamic programming.
First, in many cases, (\ref{eq:sgrad0}) can operate without knowing the transition matrices of the MDP, whereas dynamic programming
requires complete knowledge of these parameters. (This is the basis of reinforcement learning in Chapter \ref{chp:rl}.)
Second, (\ref{eq:sgrad0}) is often substantially cheaper to compute, especially for very large state and action spaces, where dynamic programming can be prohibitively expensive.

This chapter focuses  on gradient estimation for Markov processes. Algorithms discussed include
the SPSA algorithm, score function gradient estimator and weak derivative gradient estimator. Such gradient estimators
form the basis for implementing stochastic gradient algorithms for MDPs and POMDPs.
Simulation based gradient estimation is a mature area; our aim in this chapter is to present a few key algorithms that are useful
for MDPs and POMDPs.

\section{Stochastic Gradient Algorithm} \label{sec:contso}
The aim is to solve a continuous-valued stochastic optimization problem.
Suppose a parameter vector $\model \in \reals^p$ 
specifies a transition matrix (or more generally, kernel)
$\tpm$ 
and stationary distribution $\belief_{\theta}$ from which 
a Markov process  $\{\state_k\}$ is simulated. 
  The aim is to compute $\model^* \in \Model$ that minimizes\footnote{Assume $\Model \subset \reals^p$ is compact and $C(\model)$ is a continuous  function so that the minimum exists.} the expected cost
\beq C(\model) = E_{\belief_\model}\{\cost(\state,\model)\}= \int_\statespace \cost(\state,\model) \belief_\model(\state) d\state . \label{eq:stochoptobj} 
\eeq
In stochastic optimization, the stationary distribution   $\pi_\model$  and/or the cost
$\cost(\state,\model)$ are not known explicitly\footnote{We assume that $\tpm$ is known but computing 
$\belief_\model$ in closed form is not tractable. More generally, $\tpm$ itself may not be known, but for any choice of $\model$, samples of $\cost(\state_k,\th)$ can be simulated.}; instead the sample path sequence $\{c(\state_k,\model), k=1,2,\ldots\}$  of (noisy) costs
can be observed for any choice of model $\model$.
Therefore, it is not possible
to compute the integral in (\ref{eq:stochoptobj}) explicitly.
This is in contrast to a   deterministic optimization problem where $\cost(\state,\model)$ and $\belief_\model$ are known so that 
 the right hand side of (\ref{eq:stochoptobj}) can be evaluated explicitly.

In stochastic optimization, a widely used algorithm to compute (estimate) the minimizer $\model^*$ of (\ref{eq:stochoptobj})  is the following stochastic gradient algorithm run over time $k=0,1,\ldots$
\beq  \model_{k+1} = \model_k - \stepsize_k \widehat\nabla_\model \Cost_k(\model_k) . \label{eq:sgrad} \eeq
Here  $\widehat {\nabla}_\model \Cost_k(\model_k)$ denotes an   estimate
of gradient $\nablam \Cost(\model)$ evaluated at $\model_k$.  Note that  $\widehat {\nabla}_\model\Cost_k(\model_k)$  needs to be computed using the observed realization  $\{c(\state_k,\model)\}$
 since this is the only information available.

In (\ref{eq:sgrad}),  $\stepsize_k$ is a positive scalar  at each time $k$ and denotes the  step size of the algorithm. There are two general philosophies for choosing the step size.
\begin{compactenum}
\item
 Decreasing step size $\stepsize_k  = 1/{k^\alpha} $, where $\alpha \in (0.5,1]$. Under reasonable conditions, it can be proved that  the
 stochastic gradient algorithm  (\ref{eq:sgrad}) converges with probability 1 to a local stationary point of $\Cost(\model)$.
\item Constant step size $\stepsize_k = \stepsize$ where $\stepsize $ is a small positive constant.  Under reasonable conditions, it can be proved that  
the  stochastic gradient algorithm  (\ref{eq:sgrad}) converges weakly to a local stationary point.  Such constant step size algorithms are useful for estimating a time-varying minimum.\end{compactenum}

\noindent  {\em Remark. Robbins Munro Algorithm}: The stochastic gradient algorithm
(\ref{eq:sgrad}) can be viewed as a special case of the following Robbins Munro algorithm.
 Suppose one wishes  to solve numerically the following equation for $\theta$:
\beq \E_{\belief_\theta}\{ \cost(\state,\theta) \} = 0, \text{  or equivalently }
\int_\statespace \cost(\state,\theta) \,\belief_\theta(\state) d \state = 0.  \label{eq:rmaim}
\eeq
Assume that for any choice of $\th$, one can obtain 
 random (noisy) samples $\cost(\state_k,\th)$. 
The Robbins Munro algorithm for solving (\ref{eq:rmaim})  is the following stochastic approximation algorithm:
$$  \theta_{k+1} = \theta_k + \stepsize_k \, \cost(\state_k,\theta_k) .$$
Notice that   the stochastic gradient algorithm (\ref{eq:sgrad})
can be viewed as an instance of  the Robbins Munro
algorithm for solving 
$\E_{\belief_\theta}\{ \nablam \Cost_k (\model) \} = 0 $.
Chapter \ref{chp:rl} shows that the Q-learning algorithm (reinforcement learning
algorithm) for solving MDPs is also  an instance of the  Robbins Munro algorithm.

\paragraph{Examples} 
Examples of stochastic gradient algorithms include the following:

{\em 1.   Least Mean Square (LMS) algorithm}: The LMS algorithm belongs to the  class of adaptive filtering algorithms and is used widely in 
 \index{LMS algorithm}
adaptive signal processing. The objective to minimize is  \index{LMS algorithm}  \index{adaptive filtering}
  $$\Cost(\model) = \E_\belief\{ (\obs_k - \psi_k' \model)^2\}$$ where $\obs_k$ and $\psi_k$ are observed
at each time $k$. Then using (\ref{eq:sgrad}) we obtain the LMS algorithm 
\beq  \model_{k+1} = \model_ k + \stepsize \psi_k (\obs_k - \psi^\p_k\model_k) \label{eq:lms1}\eeq
Note that this case is  a somewhat simple example of a stochastic optimization problem since we know
 $\cost(\state, \model) = (\obs_k - \psi_k' \model)^2$  explicitly as a function of  $\state= (\obs,\psi)$ and  $\model$.
 Also  the measure $\belief$  does not depend explicitly on $\model$. 
In Chapter \ref{sec:consensuslms}, we will discuss the consensus LMS algorithm.

{\em 2. Adaptive Modulation}: 
Consider a wireless  communication system where the 
packet error rate  of the channel $\state_k$ evolves according to a Markov process.
The modulation scheme is adapted over time depending on the observed (empirical) packet error
rate  $\obs_k$ measured at the receiver, where $\obs_k \in [0,1]$.    
 Suppose at each time slot $k$, one of two modulation schemes can be chosen, namely, $$\action_k \in \{r_1 \text{ (low data rate)}, r_2
 \text{ (high data rate)} \}.$$
 The modulation scheme $\action_k$ affects the  transition matrix (kernel) $\tp(\action_k)$ of the error probability process
  $\state_k$.
 Define the instantaneous throughput at time $k$ as $r(\state_k,\action_k) = (1- \state_k)  \action_k $.
 The aim is to maximize the average throughput
 $\lim_{\finaltime\rightarrow \infty} \frac{1}{\finaltime+1} \E_\policy\{ \sum_{k=0}^\finaltime r(\state_k,\action_k) \} $ where
 the policy $\mu$ maps the available information $\info_k =\{u_0,\obs_1,u_1,\ldots, u_{k-1},\obs_k\}$ to action $u_k$.

The setup is depicted in Figure \ref{fig:cognitive}. The problem is an average cost POMDP and in general intractable to solve.

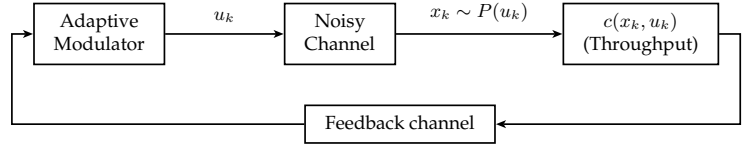
\begin{figure}[h] \centering
\scalebox{0.8}{

\begin{pspicture}[showgrid=false](0.5,-1.2)(14,1.55)


\dotnode[dotstyle=square*,dotscale=0.001](1.7,0){dot}

\psblock(2,.75){H1}{\begin{tabular}{c} Adaptive  \\ Modulator\end{tabular}}

\psblock(6,.75){B1}{\begin{tabular}{c} Noisy \\ Channel \end{tabular}}

\psblock(11,.75){B2}{\begin{tabular}{c} $\cost(\state_k,\action_k)$ \\ (Throughput)\end{tabular}}

\psblock(7,-.75){H2}{\begin{tabular}{c} Feedback channel  \end{tabular}}

\psset{style=Arrow}

\ncangle[arrows=<-,angleA=0,angleB=0]{H2}{B2}

\ncangle[angleA=180,angleB=180]{H2}{H1}


\ncline{H1}{B1} \naput[npos=.5]{$\action_k$}

\ncline{B1}{B2} \naput[npos=.5]{$\state_k \sim \tp(\action_k) $}

\end{pspicture}}
\caption{Schematic Setup of Adaptive Modulation Scheme}
\label{fig:cognitive}
\end{figure}

Consider the following ``simple'' parametrized policy for the adaptive modulation scheme:
\beq \action_{k+1} =  \policy_\model(\obs_k) =  \begin{cases} r_1 & \text{ if  } \obs_k \geq \theta  \\
					r_2  & \text{ if } \obs_k < \theta   \end{cases}  \label{eq:quantlevel}\eeq
 This family of policies $\policy_\model$ parametrized by $\model$ is  intuitive since the individual policies use a less aggressive modulation scheme (low data rate) 
 when the channel quality is poor (large error rate $\obs_k$), and vice versa.
 The aim is to determine the  value of $\th$ in the adaptive modulation policy $\policy_\model$ which maximizes the average throughput.

Let $\belief_\theta$ denote the  joint stationary probability distribution of the Markov process 
$(\state_k,\obs_{k-1})$ induced by the policy
$\policy_\model$. The aim is to determine $\th$ which maximizes the average throughput:
\beq  \text{Compute } \th^* = \argmax_{\th \in \Th}  R(\th), \;\;
\text{ where }  R(\th)  =  \E_{\belief_\theta}\{ r(\state_k, \policy(\obs_{k-1})) \}  
\label{eq:ththrough}
\eeq
given noisy measurements $r_k = (1- \obs_k) \action_k $ of the instantaneous throughput. 
An explicit formula for $\belief_\theta$ is not known  since it is a complicated function of
the modulation scheme, error correction codes, channel coding, medium access control protocol, etc. So (\ref{eq:ththrough}) is not a
 deterministic optimization problem.

Assuming $\E\{\obs_k\} = \state_k$ (the empirical observations of the channel are unbiased), then
$\th^*$ can be estimated using the stochastic gradient algorithm
$$ \th_{k+1} = \th_k + \epsilon_k\,  (\hat{\nabla}_\th r_k ).$$
Here $\hat{\nabla}_\th r_k $ is an estimate of the gradient
$\nabla_\model R(\model) = \int r(\state, \mu(\obs) ) \nabla_\model \belief_\model(\state,\obs) d\state d\obs$.
In \cite{MKS05} this approach is described in more detail.
In \secn \ref{sec:scorefn} and \secn \ref{sec:weakgrad}, two classes of simulation based algorithms are given to estimate the gradient with respect to the stationary distribution
of a Markov process. Alternatively, any of the  finite  difference methods in \secn \ref{sec:finitediff} can be used.


\section{How long to simulate a Markov chain?}

Suppose a  Markov chain is simulated for $n$-time points. Let
$\statespace = \{e_1,e_2,\ldots,e_\statedim\}$ denote the state space, where $e_i$ are the $\statedim$-dimensional unit vectors.  Given
$h\in \reals^\statedim$, let
$$ \avg_n =  \frac{1}{n} \sum_{k=0}^{n-1}  h^\p  \state_k  $$
denote the time averaged estimate.
By the strong law of large numbers, if $\tp$ is regular,
then as $n\rightarrow \infty$, $\avg_n \rightarrow h^\p \belief_\infty$ with probability 1 for any initial distribution $\belief_0$.
However, in practice, for finite sample size simulations, one needs to determine the accuracy of estimate and guidelines
for choosing the sample size $n$.
For sample size $n$, define the bias, variance and mean square deviation of the estimate $\avg_n$  as
\begin{align*}  \bias(\avg_n) &=  \E\{ \avg_n \} - h^\p  \belief_\infty  \\
\var(\avg_n) &=  \E\{\avg_n - \E\{\avg_n\} \}^2 \\
\msd(\avg_n) &=  \E\{ \avg_n - h^\p \belief_\infty\}^2 = \var(\avg_n) +   (\bias(\avg_n))^2.
\end{align*}
The main result is as follows; the proof is in \cite{Kri16}.

\begin{theorem} \label{thm:howlong}  Consider an $n$-point sample path of a  Markov chain $\{\state_k\}$ with regular transition matrix $\tp$.
Then 
\begin{align} & |\bias(\avg_n) | \leq \frac{  \max_{i,j}| {h}_i - h_j | }{n (1-\dob)} \,\dvar{\belief_0} {\belief_\infty}  \label{eq:avgexp}  \\
& \msd(\avg_n) \leq 2\, \frac{ \max_{i,j} |{h}_i - h_j|^2 }{n(1-\dob)} 
\sum_{i\in\statespace} \left(\dvar{e_i }{\belief_\infty} \right)^2 \belief_\infty(i) + O(\frac{1}{n^2})  \label{eq:msdexp} \\
&\prob \left( | \avg_n - E\{\avg_n\}| > \epsilon\right)  \leq 2 \exp\left( - \frac{ \epsilon^2  (1 - \dob)^2 n}{ \max_{l,m} |h_l - h_m|^2 } \right).
\label{eq:concmeasure}
\end{align}
Here, for any two pmfs $\belief,\bar{\belief}$,
$\dvar{\belief}{\bar{\belief}}$ denotes the variational distance
$$
\dvar{\belief}{\bar{\belief}} =  \frac{1}{2} \|\belief - \bar{\belief}\|_1 = \frac{1}{2} \sum_i |\belief(i) - \bar{\belief}(i)|
$$
and $\dob$ denotes the Dobrushin coefficient of  transition matrix $\tp$
$$
\dob =  \frac{1}{2}  \max_{i,j}  \sum_{l\in \statespace} | \tp_{il}- \tp_{jl} |.
$$
\end{theorem}

\section{Gradient Estimators for Markov Processes}
Consider the stochastic optimization problem (\ref{eq:stochoptobj}). 
To  solve (\ref{eq:stochoptobj}) using the
   stochastic gradient algorithm (\ref{eq:sgrad})  requires
estimating the gradient $\nabla_\model C(\model)$.  
We wish to compute 
the gradient estimate  $\hat \nablam \Cost_\bsize(\model)$  using  the observed realization of costs $\{c(\state_k,\model)\}$, $k =0,\ldots,\bsize$.
Here $\state_k$ is an ergodic  Markov process with transition matrix  or kernel $\tpm$.
The rest of this chapter discusses  two  types of  simulation based gradient estimators for Markov processes:

\begin{compactitem}
\item Finite difference gradient estimators such as the Kiefer Wolfowitz and SPSA algorithms discussed in \chp  \ref{sec:finitediff}. 
They do not require knowledge of the transition kernel of the underlying Markov process.  However,
finite difference gradient estimators
suffer from a bias variance tradeoff.   
\item  Gradient estimators that exploit knowledge of the  Markov transition kernel. We discuss the 
score function gradient estimator  in \chp  \ref{sec:scorefn} and the weak derivative estimator below  in \chp  \ref{sec:weakgrad}.
 Unlike the finite difference estimators,  these gradient estimators are unbiased for random variables.
\end{compactitem}
 \chp  \ref{chp:rl} shows how these gradient estimators can be used for reinforcement learning of MDPs, constrained MDPs, and POMDPs.

\section{Finite difference gradient estimators and  SPSA} \label{sec:finitediff}

This section  describes two finite difference gradient estimators
that can be used with the stochastic gradient algorithm to minimize (\ref{eq:stochoptobj}).
The algorithms do not require the transition matrix of the Markov chain to be known.


\subsection{Kiefer Wolfowitz algorithm}  This  uses the two sided numerical approximation
of a derivative and is described in Algorithm \ref{alg:kwbasic}. 
Here $\Delta_n$ is the gradient step while $\epsilon_n$ is the step size of the stochastic approximation algorithm.
These need to be chosen so that
\beq \Delta_n > 0, \; \epsilon_n > 0, \; \Delta_n \rightarrow 0, \; \epsilon_n \rightarrow 0,\;  \sum_{n} \epsilon_n = \infty,\;  \sum_{n} \epsilon_n^2/ \Delta_n^2 < \infty. \label{eq:stepsizeseq} \eeq

A disadvantage of the Kiefer Wolfowitz algorithm is that 
 $2p$ independent simulations are required  to evaluate the gradient along  all the possible directions $i=1,\ldots,p$. 

\begin{algorithm}
\caption{Kiefer Wolfowitz  Algorithm ($\theta_n \in \reals^p$)} \label{alg:kwbasic}
 For iterations $n=0,1,2,\ldots$ 
\begin{itemize}
\item Evaluate the $2p$ sampled costs $ \hat{\Cost}_n(\model+\Delta_n e_i)$ and $\hat{\Cost}_n(\model - \Delta_n e_i)$, $i=1,2,\ldots,p$ where
 $e_i$ is a unit vector with 1 in element $i$. 
\item Compute gradient estimate   \beq \hat{\nabla}_{\model(i)} \Cost_n(\model) = \frac{1}{2\Delta_n}
\left[\hat{\Cost}_n(\model+\Delta_n e_i) - \hat{\Cost}_n(\model - \Delta_n e_i)\right],
\quad i=1,\ldots, p \label{eq:fdderiv} \eeq
 Here $\Delta_n= \frac{\Delta}{(n+1)^\gamma}$ denotes the gradient step size with $0.5 \leq \gamma \leq 1$ and $\Delta > 0$.
\item Update model estimate  $\model_n$  via 
stochastic gradient  algorithm
$$
\model_{n+1} = \model_n - \epsilon_{n}  \hat{\nabla} C(\model_n), \quad
\epsilon_n = \frac{\epsilon}{(n+1+s)^\zeta}, \quad 0.5 < \zeta \leq  1, \;
\text{ and }\epsilon , s > 0. $$
\end{itemize}
\end{algorithm}

\subsection{Simultaneous Perturbation Stochastic Approximation (SPSA)}   \index{SPSA algorithm}
The SPSA algorithm \cite{Spa03} has been pioneered  by J. Spall (please see  the website {\tt www.jhuapl.edu/SPSA/}). It overcomes the problem of  requiring $2p$ independent simulations 
in the Kiefer Wolfowitz algorithm by 
choosing  a single random direction in $\reals^p$ along which
to evaluate the finite difference gradient.
So only  2 simulations are required. 
The SPSA algorithm is described in Algorithm \ref{alg:spsabasic}.
The step sizes $\Delta_n$ and $ \epsilon_n$ are chosen as in (\ref{eq:stepsizeseq}).

In  Algorithm \ref{alg:spsabasic} the random directions have been chosen from a Bernoulli distribution. 
The random directions $d_n$ need to be chosen from a distribution such that 
the inverse moments $\E\{|1/d_n(i)|^{2+2\alpha}\}$ are finite for some $\alpha > 0$.
Suitable choices (apart from Bernoulli) include
segmented uniform and U-shaped densities; see  \cite[pp.185]{Spa03} for a detailed description.

Spall shows that SPSA is asymptotically as efficient as Kiefer Wolfowitz.

\begin{algorithm}[h]
\caption{SPSA Algorithm   ($\theta_n \in \reals^p$)} \label{alg:spsabasic}
 For iterations $n=0,1,2,\ldots$ 
\begin{itemize}
\item Simulate  the $p$ dimensional  vector $d_n$ with random elements 
$$d_n(i) = \begin{cases} 
-1 & \text{ with probability } 0.5 \\
+1 & \text{with probability } 0.5.\end{cases} $$
\item Evaluate sample costs $\hat{C}_n(\model_n+\Delta_n d_n)$ and $\hat{C}_n(\model_n-\Delta_n d_n)$
\item Compute gradient estimate   
$$  \hat{\nabla} C(\model_n) = \frac{\hat{C}_n(\model_n+\Delta_n d_n) - \hat{C}_n(\model_n-\Delta_n d_n)}{\displaystyle 2
\Delta_n } d_n.
$$
 Here $\Delta_n= \frac{\Delta}{(n+1)^\gamma}$ denotes the gradient step size with $0.5 \leq \gamma \leq 1$ and $\Delta > 0$.

\item Update model estimate  $\model_n$  via 
stochastic gradient  algorithm
$$ 
\model_{n+1} = \model_n - \epsilon_{n}  \hat{\nabla} C(\model_n), \quad
\epsilon_n = \frac{\epsilon}{(n+1+s)^\zeta}, \quad 0.5 < \zeta \leq  1, \;
\text{ and }\epsilon , s > 0.  $$ 
\end{itemize}
\end{algorithm}

Finite difference methods such as SPSA
suffer from the bias-variance tradeoff.
The bias
 in the gradient estimate is proportional to $\Delta^2$. 
On the other hand, if  $\hat{\Cost}_n(\model-\Delta e_i)$ 
and $\hat{\Cost}_n(\model+\Delta e_i)$ are sampled independently, then the variance
is proportional to $1/\Delta^2$.  To decrease the bias, one needs a small $\Delta$, but this results in an increase in the variance.


\section{Score Function Gradient Estimator} \label{sec:scorefn}   \index{gradient estimator! score function}
This section describes the score function gradient estimator. Unlike the SPSA algorithm, no finte difference approximation
is used. We  assume that the transition
kernel  $\tpm$ of the Markov process $\state_k$ is known but computing the stationary distribution 
$\beliefm$ is intractable. The aim is to compute the gradient estimate $\hat\nabla_\model \Cost_n(\model_n)$ in the stochastic
gradient algorithm (\ref{eq:sgrad}) given the realization of costs $\{\cost(\state_k,\model)\}$, $k \in \batch_n$ for any choice of $\model$.

\subsection{Score Function Gradient Estimator  for RVs} \label{sec:rvgrad}
To highlight the main ideas,
the score function gradient  estimator for random variables is discussed first.

 Assuming sufficient regularity
   to swap the order of integration and gradient, we have
 $$\nabla_\model C(\model) = \int c(X) \nabla_\model \pi_\model(X) dX
=\int c(X)\frac{ \nabla_\model \pi_\model(X)}{\pi_\model(X)} \pi_\model(X) dX$$
The score function algorithm proceeds as follows:  simulate $X_k \sim \pi_\model$ and compute for any $N$
$$\hat{\nabla_\model} C_\finaltime = 
\frac{1}{N} \sum_{k=1}^N
c(X_k)\frac{ \nabla_\model \pi_\model(X_k)}{\pi_\model(X_k)}
$$
The term ``score function'' stems from the fact that $\frac{ \nabla_\model \pi_\model(X)}{\pi_\model(X)}= \nabla_\model \log  \pi_\model $.
For any $\finaltime$, this is an unbiased estimator of $\nabla_\model C(\model)$.
The derivative of the log of density function is often called the ``score function'' in statistics.

\noindent {\bf Example}: If $\pi_\model(x) = \model e^{-\model x}$, then 
$\frac{ \nabla_\model \pi_\model(X)}{\pi_\model(X)} = 
\nabla_\model \log  \pi_\model = \frac{1}{\model} - X$.

\subsection{Score Function Gradient Estimator for Markov Process}
We now describe the score function simulation based
gradient estimator for a  Markov process. The  eventual goal is to solve finite state MDPs and POMDPs via
stochastic gradient algorithms. So
to 
avoid technicalities,
we  consider a finite state Markov chain $x$ and assume that transition matrix $\tpm$ is regular so that
there exists a unique  stationary distribution $\beliefm$. Therefore, for a cost $c(x)$, the expected cost is
%
$$\E_{\beliefm}\{ \cost(\state)\} = \sum_{i=1}^\statedim \cost(i)\, \beliefm(i) = \cost^\p \beliefm. $$
Suppose that for  any choice of $\th$, one can observe via simulation a sample path of the costs $c(\state_k)$, $k=1,2,\ldots$ where the
Markov chain
$\{x_k\} $ evolves with transition matrix $\tpm$.
Given such simulated sample paths, the aim is to estimate  the gradient
$$\nabla_\model\E_{\beliefm}\{c(\state)\} = \nabla_\model  \left( \cost^\p \belief_{\model} \right) .$$ 
This is what we mean by simulation based gradient estimation.
We  assume that  the transition matrix $\tpm$ is known but the stationary distribution $\beliefm$ is not known explicitly.  Also
the cost.  $\cost(\state)$ may not be known explicitly and the simulated sampled path $c(\state_k)$ may be observed in zero mean noise.
%

Clearly as $\bsize\rightarrow \infty$, for any initial distribution  $\belief_0$,
$ \lim_{\bsize \rightarrow \infty} \cost^\p \tpm'^\bsize  \belief_0 \rightarrow \cost^\p \beliefm $.  
%
So the finite sample approximation is (write $\tpm$ as $\tp$ to simplify notation)
\begin{align}
 \nabla_\model  \left( \cost^\p \belief_{\model} \right) &\approx \belief_0^\p \sum_{k=0}^{\bsize-1} \tp^{\bsize-k-1}\,  \nablam \tp \, \tp^k\,  \cost \label{eq:finitesf}\\
&= \sum_{x_{0:\bsize}}  \cost(\state_\bsize) \biggl[ (\nablam \tp_{\state_0\state_1}) \tp_{\state_1\state_2}\cdots \tp_{\state_{\bsize-1}\state_\bsize}
+  \tp_{\state_0\state_1} (\nablam  \tp_{\state_1\state_2}) \cdots \tp_{\state_{\bsize-1}\state_\bsize}  \nn \\ &  \hspace{2cm} + \cdots + 
  \tp_{\state_0\state_1}   \tp_{\state_1\state_2} \cdots (\nablam \tp_{\state_{\bsize-1}\state_\bsize})
\biggr] \belief_{\state_0}  \nn \\
&= \sum_{x_{0:\bsize}} \cost(\state_\bsize) \left[ \frac{\nablam \tp_{\state_0\state_1}}{ \tp_{\state_0\state_1}} + \cdots + \frac{\nablam \tp_{\state_{\bsize-1}\state_\bsize}}{ \tp_{\state_{\bsize-1}\state_\bsize}} \right] \tp_{\state_0\state_1}  \tp_{\state_1\state_2} \cdots  \tp_{\state_{\bsize-1}\state_\bsize} \belief_{x_0}  \nn \\
&= \E_{\state_{0:\bsize}} \left\{  \cost(\state_\bsize) \,  \sum_{k=1}^\bsize \frac{ \nablam \tpm_{\state_{k-1}\state_{k}}}
{  \tpm_{\state_{k-1}\state_{k}}} \right\} \nn
 \end{align}

This leads to the score function gradient estimator in Algorithm \ref{alg:score}.

\begin{algorithm}\caption{Score Function Gradient Estimation Algorithm} \label{alg:score}
\begin{compactenum}
\item[Step 1.] Simulate Markov chain $\state_0,\ldots,\state_{\bsize}$ with transition matrix  $\tpm$.
\item[Step 2.] Compute  $ S^\model_\bsize =   \sum_{k=1}^\bsize \frac{ \nablam \tpm_{\state_{k-1}\state_{k}}}
{  \tpm_{\state_{k-1}\state_{k}}}  $ 
This is evaluated recursively as $$S^\model_k =  \frac{ \nablam \tpm_{\state_{k-1}\state_{k}}}
{  \tpm_{\state_{k-1}\state_{k}}} + S^\model_{k-1},\quad  k=1,\ldots, \bsize.$$
\item[Step 3.]  Evaluate the score function gradient estimate  via simulation as 
$$\widehat{\nablam} C_\bsize(\model) = \frac{1}{\bsize} \sum_{k=1}^{\bsize} \cost (\state_k) S^\model_k.$$ (Note that this is the average over $\bsize$ samples).

\end{compactenum}
\end{algorithm}

\section{Weak Derivative Gradient Estimator} \label{sec:weakgrad}   \index{gradient estimator! weak derivative}
This section describes the weak derivative  gradient estimator. Like the score function estimator, for random variables it provides an unbiased
estimator of the gradient.

A probability distribution $\cdf_\model$  is weakly differentiable at a point $\model$ (on an open subset of $\reals)$ if there exists
a signed measure\footnote{Let $\Omega$ be the set of outcomes and $\mathcal{A}$ denote a sigma-algebra defined on $\Omega$.
Then 
 a signed measure $\nu$ is a  real valued function on  $\mathcal{A}$ which satisfies: (i) $\sigma$-additive - meaning that if
$A_i,i=1,2,\ldots$ are disjoint sets (events) in $\mathcal{A}$,  then $\nu (\cup_i A_i) = \sum_i \nu(A_i)$.  (ii) $\nu(\emptyset) = 0$. \\
A signed measure is finite if $\nu(\Omega) < \infty$.  Note a non-negative signed measure is called a measure.
Finally, $\nu$ is a probability measure if 
 $\nu(\Omega) = 1$.}
 denoted as $\nu_\model = \nabla_\model F$ such that 
\beq
\lim_{s\rightarrow 0} \frac{1}{s} \left[\int c(\state) d F_{\model + s} (\state) - \int c(\state) dF_\model(\state) \right] = \int c(\state) d\nu_\model (\state)\
\label{eq:weakderivativedefn}
\eeq for
 any bounded continuous function $c(\state)$.
The term weak derivative is used since $\nabla_\model \cdf(\state)$ may not be a function in the classical sense; it could be a  generalized function (e.g. Dirac delta function). Hence the above
 definition involves integration with a test function $\cost(\state)$.

A well known result in measure theory is that any finite signed measure $\nu$ can be decomposed as 
\beq \nu = \wderiv_1 \mu_1 - \wderiv_2 \mu_2
\label{eq:wddecomp0}
\eeq where $\wderiv_1$ and $\wderiv_2$ are constants and $\mu_1$ and $\mu_2$ are probability measures.
In our case, since the signed measure is  obtained as the derivative of a probability measure,  i.e., since  $\int d\cdf_\model(x) = 1$, therefore $\int  d \nu_\model (x) = 0$, implying that $\wderiv_1 = \wderiv_2$ in (\ref{eq:wddecomp0}).
So the definition of weak derivative of a probability distribution can be re-expressed as:
\begin{definition}
A probability distribution $\cdf_\model$  is weakly differentiable at a point $\model$ (on an open subset of $\reals)$ if there exist
probability distributions $\dot{\cdf}_\model$ and  $\ddot{\cdf}_\model$ and a constant $\wderiv_\model$ such that for
 any bounded continuous function $c(\state)$, the following holds:
 $$  \nabla_\model \int c(\state)  d\cdf_\model(\state) =  \wderiv_\model \left[ \int c(\state)  d \dot{\cdf}_\model(\state)  - \int c(\state) d\ddot{\cdf}_\model
 (\state) \right] .$$ \end{definition}
In more familiar engineering notation using probability density functions, the definition of the weak derivative is
\beq   \nabla_\model \int c(\state)  \pdf_\model(\state) =  \wderiv_\model \left[ \int c(\state)   \dot{\pdf}_\model(\state)  d\state- \int c(\state) \ddot{\pdf}_\model(\state) 
d \state) \right] \label{eq:wdpdf} \eeq
where $\dot{\pdf}_\model$ and $\ddot{\pdf}_\model$ are probability density functions,
or equivalently,
\beq \nabla  \pdf_\model = \wderiv_\model (  \dot{\pdf}_\model - \ddot{\pdf}_\model).  \label{eq:wddef2}\eeq
The weak derivative is specified by the triplet $(\wderiv_\model, \dot{\pdf}_\model,\ddot{\pdf}_\model)$.
A similar characterization holds in terms of probability mass functions.

In general  the  representation (\ref{eq:wddef2}) of the weak derivative is not unique. 
One specific representation of interest is obtained via the so called Hahn Jordan decomposition. This is a deep result in measure
theory - for our practical needs, the following simplistic version suffices:
A signed measure can 
be decomposed as in (\ref{eq:wdpdf}) such that densities $\dot{\pdf}_\model$ and $\ddot{\pdf}_\model$ are orthogonal. This means that the set of $x$ where density
$\dot{p}_\model(x)$ is non-zero coincides with the set of $x$ where density $\ddot{p}_\model(x)$ is zero, and vice versa.

 \subsection{Weak derivative of random variables}
Based on (\ref{eq:wddef2}), the weak derivative  gradient estimation for  a random variable is  as follows: Simulate $\finaltime$ samples
   $\dot{X}_k \sim \dot{\pdf}_\model$ and $\ddot{X}_k \sim \ddot{\pdf}_\model$, $k=1,2, \ldots, \finaltime$. Then 
 $$\hat{\nabla_\model} C_\finaltime = \frac{1}{N} \sum_{k=1}^\finaltime \wderiv_\model \left[  c(\dot{X}_k) - c(\ddot{X}_k) \right] .$$
 
 \noindent {\bf Examples}: Here are  examples of the weak derivative of random variables: 
\\
1. {\em Exponential}: 
 $\pi_\model(x) = \model e^{-\model x}$. Then the Hahn-Jordan decomposition is
 $$\nabla_\model\pi_\model(x) = e^{-\model x}(1 - x \model) I(x>\frac{1}{\model}) 
 - e^{-\model x}(x \model- 1) I(x\leq\frac{1}{\model}) $$
 So 
$\wderiv_\model = \model e$, $\dot{\pdf}_\model(x) = \model e ^{-\model x}(1 - x \model) I(x>\frac{1}{\model})$, $\ddot{\pdf}_\model(x) = \model e ^{-\model x}(x \model-1) I(x\leq \frac{1}{\model})$.
 
 \noindent 2. {\em Poisson}: For a Poisson random variable, the probability mass function is
 $$\pdf_\model(x) =  \frac{e^{\model} \model^x}{x!} , \quad x = 0,1,\ldots $$
So clearly, 
$ \nabla_\model \pdf_\model(x) = \pdf_\model(x-1) - \pdf_\model(x). $
So one possible weak derivative implementation  is $$\wderiv_\model = 1,\; \dot{\pdf}_\model = \pdf_\model(x-1),\; \ddot{\pdf}_\model = \pdf_\model(x).$$

\subsection{Weak Derivative Gradient Estimator for Markov Process}
Let $\beliefm$ denote  the stationary distribution of a Markov chain with
regular transition matrix $\tpm$. Suppose that for any choice of $\th$, one can observe via simulation  sample paths 
$\{\cost(\state_k)\}$ of the costs where Markov chain $\state_k$ evolves with transition  matrix $\tpm$.
The aim is to estimate via simulation the  gradient
$\nabla_\model\E_{\belief_\model}\{c(\state)\} = \nabla_\model  \left(\beliefm^\p    c \right) $.
 It is assumed that $\tpm$ is known but $\beliefm$ is not known explicitly. Also $\cost(\state)$ may not be known explicitly and the simulated sampled path $c(\state_k)$ may be observed in zero mean noise.

Analogous  to  random variable case (\ref{eq:wdpdf}), define the weak derivative of the transition
matrix $\tpm$ as follows:  A transition probability matrix $\tpm$ is weakly differentiable at a point $\theta$ (on an open subset of $\reals$)
if there exists transition probability matrices $\dot{\tpm}$ and  $ \ddot{\tpm}$  and a diagonal matrix  $\wderiv_\model$ such that: 
\begin{align} 
\nabla_\model \tpm  &= \wderiv_\model ( \dot{\tpm} - \ddot{\tpm}) , \label{eq:wdmarkov} \\
\text{ equivalently, }  & \nabla_\model \tpm_{ij} =  \wderiv_{\model, i} \,(\dot{\tpm}_{ij} -
 \ddot{\tpm}_{ij} ) , \quad i,j \in \{1,2,\ldots,\statedim\}.\nn
\end{align}
So the  weak derivative of  transition matrix $\tpm$ is specified by the triplet $(\dot{\tpm}, \ddot{\tpm}, \wderiv_\model)$.
Obviously, $\nablam \tpm \one = \nablam \one = 0$ implying that each row of $\nablam\tpm$ adds to zero.

\begin{theorem} For   regular transition matrix $\tpm$ with stationary distribution $\beliefm$ and cost vector $\cost \in \reals^\statedim$,
\beq \nabla_\model  \left(\belief_\model^\p    c \right) =  \belief_\model^\p  (\nabla_\model \tpm ) \sum_{k=0}^\infty \tpm ^k  c 
\label{eq:wdexact}
\eeq
\end{theorem}

The above theorem leads to the following finite sample estimate:
\beq \nabla_\model  \left(\belief_\model^\p    c \right) \approx  \belief_0^\p (\tpm)^m  (\nabla_\model \tpm )  \sum_{k=0}^\bsize (\tpm)^k c \label{eq:wdapprox}
\eeq for any initial distribution $\belief_0$.
For sufficiently large $m$ and $\bsize$ (\ref{eq:wdapprox}) will approach (\ref{eq:wdexact}). The sample path interpretation of (\ref{eq:wdapprox})  leads to 
the weak derivative estimation  Algorithm \ref{alg:weak}. Step 1 simulates  $\belief_0^\p (\tpm)^m $, Step 2 implements the weak derivative $(\nabla_\model \tpm )$ and propagates
 the two chains $\dot{x}$ and $\ddot{x}$ for $\bsize$ steps. Finally Step 3 estimates the right hand side of (\ref{eq:wdapprox}). So Step 3 yields that
  $$ \widehat{\nabla}_\model C_\bsize(\model)  = \wderiv_{\model,\state_m} \sum_{k=m}^{\bsize} c(\dot{\state}_k) - c(\ddot{\state}_k) .$$
 
 Note if the Markov chains $\dot{\state}$ and $\ddot{\state}$ are simulated with common random numbers, then
  at some time point $\tau$,  $\dot{\state} _\tau = \ddot{\state}_\tau$,  and the processes $\dot{x}$,
 $\ddot{x}$ merge and evolve identically after time~$\tau$. This yields Step 3 of Algorithm~\ref{alg:weak}.

\begin{algorithm}\caption{Weak Derivative based Gradient Estimation Algorithm} \label{alg:weak}
Evaluate triplet $(\dot{\tpm}, \ddot{\tpm}, \wderiv_\model)$ using formula $ \nabla_\model \tpm_{ij} =  \wderiv_{\model,i}\, (\dot{\tpm}_{ij} - \ddot{\tpm}_{ij} ) $.
\begin{compactenum}
\item[Step 1.] Simulate  the Markov chain $\state_0,\ldots\state_{m-1}$ with transition matrix  $\tpm$.
\item[Step 2a.] Starting with state $\state_{m-1}$, simulate at time $m$ the states  $\dot{\state}_{m}$
and $\ddot{\state}_{m}$ with  transition matrix $\dot{\tpm}$ and $\ddot{\tpm}$, respectively.
\item[Step 2b.]  Starting with states $\dot{\state}_{m}$ and $\ddot{\state}_m$, respectively,
 simulate the two Markov chains $\dot{x}_{m+1}, \dot{x}_{m+2},\ldots,\dot{x}_\bsize $  and  $\ddot{x}_{m+1}, \ddot{x}_{m+2},\ldots,
 \ddot{x}_\bsize $ with transition matrix ${\tpm}$. \\
Use  the same random numbers to simulate these Markov chains 
\item[Step 3.]  Evaluate the weak derivative estimate via simulation as  
\beq 
\widehat{\nabla}_\model C_{m,\bsize}(\model)  = \wderiv_{\model,\state_m} \sum_{k=m}^{m+\tau} c(\dot{\state}_k) - c(\ddot{\state}_k) ,
\text{ where } \tau = \min\{k: \dot{x}_k = \ddot{x}_k, k \leq \bsize\}.  \label{eq:wddecomp}
\eeq
\end{compactenum}
\end{algorithm}

\section{Bias and Variance of Gradient Estimators} \label{sec:biasvarwd}
Here we characterize the statistical properties of the score function and weak derivative gradient estimators 
discussed above.

\begin{theorem} \label{thm:scoremarkov}
For a Markov chain with initial distribution $\belief_0$,  regular transition matrix $\tpm$ with coefficient of ergodicity $\dob$   and stationary distribution $\beliefm$:
\begin{compactenum}
\item The score function gradient estimator of Algorithm \ref{alg:score} has: 
\begin{compactenum}
\item Bias: $ \E\{\widehat{\nablam } C_\bsize \}-  \nablam \cost^\p \beliefm = O(1/\bsize)$.   
\item Variance: $\var\{\widehat{\nablam } C_\bsize \} = O(\bsize) $ 
\end{compactenum}
\item The weak derivative gradient estimator of Algorithm \ref{alg:weak} has: 
\begin{compactenum}
\item Bias: $ \E\{\widehat{\nablam } C_{m,\bsize} \} - \nablam \cost^\p \beliefm= O(\dob^m) \, \dvar{\belief_0}{\beliefm} + O(\dob^\bsize) $
\item Variance: $\var\{\widehat{\nablam } C_\bsize \} = O(1) $ 
\end{compactenum}
\end{compactenum}
\end{theorem}
The proof is in \cite{Kri16}.
The result shows that  despite the apparent simplicity of the score function gradient estimator (and its widespread use), the weak derivative estimator performs better in both bias and variance.
The variance of the score function estimator  actually
grows with sample size! We will show in numerical examples for reinforcement learning of MDPs in \chp  \ref{sec:numericalwd}, that the weak derivative estimator has a substantially smaller variance
than the score function gradient estimator.

Why is  the variance of the score function gradient estimator $O(N)$ while the variance of the weak derivative estimator is $O(1)$?
The weak derivative estimator uses the difference of two sample paths, $\{\dot{x}_m\}$ and $\{\ddot{x}_m\}$, $m=1,2\ldots,N$. Its variance
is dominated by a term of the form $\sum_{m=1}^N g^\p (\tpm^m)^\p (\belief_0 - \bbelief_0)$ where $\belief_0$ and $\bbelief_0$ are the
distributions  of $\dot{x}_0$ and $\ddot{x}_0$. This sum is bounded by $\text{constant} \times  \sum_{m=1}^N \rho^m $ which is $O(1)$ since  $\rho < 1$ for a regular
transition matrix $\tpm$.
In comparison, the score function estimator uses a single sample path and its variance is dominated by a term of the form  $\sum_{m=1}^N g^\p (\tpm^m)^\p \belief_0 $.
This sum grows  as $O(N)$. 

\section{\pwe}

This chapter  presented an introductory  description of simulation based gradient estimation and is not meant to
be a comprehensive account. The proofs and further details appear in \cite{Kri16}. We have discussed only {\em scalar} step size stochastic approximation algorithms.

The books \cite{Hay13,Say08,SK95} present  comprehensive accounts
 of adaptive filtering. Adaptive filtering constitutes an important class of  stochastic approximation algorithms - \cite{BMP90,KY03} are {\em the} books
 in the analysis of stochastic approximation algorithms. The SPSA algorithm was pioneered by Spall; see \cite{Spa03} and the website
\url{http://www.jhuapl.edu/spsa/} for repository of code and references. \index{SPSA algorithm}
  \cite{Pfl96} is an excellent book for
coverage of simulation based gradient estimation; indeed, \secn \ref{sec:scorefn} and \secn \ref{sec:weakgrad} are based on \cite{Pfl96}.

In the  discussion of gradient estimators we have omitted the important topic of 
infinitesimal perturbation analysis  (IPA) and process derivatives.

\chapter{Reinforcement Learning}  \label{chp:rl}

\minitoc 
 Stochastic dynamic programming
assumes that the MDP or POMDP model is completely specified.
 This chapter presents simulation-based  stochastic approximation algorithms for estimating  the
 optimal policy of MDPs when  the transition probabilities are not known. {\em Simulation-based}  means that although
the transition probabilities 
 are unknown,  the decision maker 
can observe the system trajectory under any choice of control actions.
 The simulation-based algorithms given in this chapter also apply as  suboptimal methods for solving  POMDPs.

The following  algorithms are presented in this chapter:
\begin{compactenum}
\item  The Q-learning algorithm is described in \secn \ref{sec:qlearning}.
It uses the Robbins Munro algorithm (described in Chapter \ref{chp:simulation}) to estimate the value function for an  unconstrained MDP.
It is also shown how a primal-dual Q-learning algorithm can be used for MDPs with monotone optimal policies.
The Q-learning algorithm also applies as  a suboptimal method for POMDPs.
\item Policy gradient 
 algorithms are then presented.  These use   gradient estimation (described
  in Chapter \ref{chp:simulation}) of the cost function
together with a stochastic gradient algorithm to estimate the optimal policy. The policy gradient algorithms apply to MDPs and constrained MDPs. They also   \index{policy gradient algorithm}
 yield a suboptimal policy search method for POMDPs. 
 \end{compactenum}
   
Some terminology: Determining the optimal policy of an MDP or POMDP when the parameters are not known falls under the class of stochastic 
{\em adaptive} control problems.
Stochastic adaptive control algorithms are of  two types: {\em direct
  methods}, where the unknown transition probabilities $\tp_{ij}(\action)$ 
are estimated
 simultaneously while updating the control policy, and {\em
  implicit methods}  (such as simulation based methods), where the
transition probabilities are not directly estimated in order to
compute the control policy.
In this chapter, we focus on implicit simulation-based algorithms
for solving   MDPs and POMDPs. These are also called {\em reinforcement learning algorithms}.
One motivation for such implicit methods is that since they are simulation based, only regions of the state space visited by the simulated sample
path are used to determine the controller. Effort is  not wasted on determining parameters for low probability regions which are rarely
or never visited.

\section{Q-learning Algorithm} \index{Q-learning}  \label{sec:qlearning}

The Q-learning algorithm is a widely used reinforcement learning algorithm.
%
 As described below, the Q-learning algorithm  is the  simply the Robbins Munro stochastic approximation
algorithm  (\ref{eq:rmaim})  of Chapter \ref{chp:simulation}  applied to estimate the value function of Bellman's dynamic programming equation.

\subsection{Discounted Cost MDP} \label{sec:qld}
Bellman's dynamic programming equation for a discounted cost MDP (with discount factor $\discount$) reads
\begin{align}
\valuef(i) &= \min_{\action\in \actionspace}  \left( c(i,\action) + \discount\, \sum_j \tp_{ij}(\action) \valuef(j)\right)  \nonumber \\
&=   \min_{\action\in \actionspace} \left( c(i,\action) + \discount\, \E\{ \valuef(\state_{k+1})| \state_k = i, \action_k = u\} \right) \label{eq:valueq}
 \end{align}
 For each state action pair $(i,\action)$  define the Q-factors as
 $$Q(i,\action) = c(i,\action) + \discount\, \sum_j \tp_{ij}(\action) \valuef(j), \quad i \in \statespace, u \in \actionspace.$$ 
 We see from (\ref{eq:valueq}) 
that the Q-factors  $Q(i,\action)$ can be expressed as
\begin{align}
Q(i,\action) &= c(i,\action) +\discount\,  \sum_j \tp_{ij}(\action) \min_{u'} Q(j,u') \label{eq:ql1}\\
&= c(i,\action) +\discount\, \E\{ \min_{u'} Q(\state_{k+1},u') | \state_k = i, \action_k = u) \}. \label{eq:qupdate}
\end{align}
 Equation  (\ref{eq:valueq}) has an expectation $ \E$  inside
the minimization, whereas (\ref{eq:qupdate}) has the expectation {\em outside} the minimization. It is this crucial observation that forms  the basis
for using stochastic approximation algorithms to estimate the Q-factors.  Indeed  (\ref{eq:qupdate}) is simply  
of the form
\beq  \E\{ f(Q) \} = 0  \label{eq:rmq}\eeq
where the random variables $f(Q)$ are defined as
\beq f(Q) =    c(\state_k,\action_k) +  \discount\, \min_{u'} Q(\state_{k+1},u')  - Q(\state_k,\action_k) . \label{eq:fdefq}\eeq
The Robbins Munro algorithm ((\ref{eq:rmaim}) in \chp   \ref{sec:contso})  can be  used to estimate the solution $Q^*$
of (\ref{eq:rmq}) as follows: Generate a sequence of estimates $\hQ_k$ as
\beq  \hQ_{k+1}(\state_{k}, \action_{k})  = \hQ_k(\state_k,\action_k)  + \epsilon_k f(\hQ_k)   \label{eq:qlearning}\eeq
The step size $\epsilon_k$ is chosen\footnote{In general, the decreasing step size of a stochastic approximation algorithm
needs to satisfy $\sum_k \epsilon_k =\infty$ and $\sum_k \epsilon_k^2 < \infty$} as
 \beq \epsilon_k =  \frac{\epsilon}{\text{Visit}(i,\action,k)}  \label{eq:qlearnss}\eeq
where $\epsilon$ is a positive  constant and $\text{Visit}(i,\action,k)$ is the number
of times the state-action pair $(i,\action)$ has been visited until time $k$ by the
algorithm.

\begin{algorithm}[h]
\caption{Q-learning Algorithm} \label{alg:qlearning}
For $n=0,1,\ldots, $ (slow time scale): \\
\mbox{} \hspace{0.5cm} 
Update policy as $\policy_n(i) =  \min_{\action \in \actionspace} \hQ_{n\Delta}(i,\action)$ for $i=1,2,\ldots,\statedim$.
\\
\mbox{} \hspace{1cm} For $k= n \Delta , n\Delta + 1,\ldots, (n+1)\Delta -1$   (fast time scale)\\
\mbox{} \hspace{1.5cm}
Given state $\state_k$, 
choose action $\action_k = \policy_n(\state_k)$. \\
\mbox{}  \hspace{1.5cm} 
Simulate next state $x_{k+1} \sim \tp_{\state_k,\state_{k+1}}(\action_k)$. \\
\mbox{}  \hspace{1.5cm} Update Q-factors as  (step size $\epsilon_k$ is chosen according to (\ref{eq:qlearnss}))
$$   \hQ_{k+1}(\state_k, \action_k)  = \hQ_k(\state_k,\action_k)  + \epsilon_k \left[ c(\state_k,\action_k) +  
\discount\,\min_{u'} \hQ_k(\state_{k+1},u')  - \hQ_k(\state_k,\action_k) \right] $$
\end{algorithm}

 Algorithm~\ref{alg:qlearning} summarizes the entire
procedure as a two-time scale stochastic approximation algorithm.
On the fast time scale,
 the Q factors
are updated applying the same policy for a fixed period of time
slots referred to as {\em update interval}, denoted as $\Delta$ in Algorithm \ref{alg:qlearning}.  After that $n$-th interval, the
new policy $\policy_{n+1}(i)$ is chosen based on current Q-factors
as  $\policy_{n+1}(i) = \min_{\action \in \actionspace} \hQ_{(n+1)\Delta}(i,\action)$. This update is done on the slow time scale.

Note that  Q-learning does not require explicit knowledge of the transition probabilities - all that is needed is access to the 
controlled system so as to measure its next state $\state_{k+1}$ when an action $\action_k$ is applied.
 For a 
finite state MDP, Q-learning algorithm converges with
probability one to the optimal solution of Bellman's equation; see \cite{KY03,BT96} for conditions and proof. Please refer to
\cite{BY12} for novel variations of the Q-learning algorithm for discounted cost MDPs.

\subsection{Primal-Dual Q-learning for Submodular MDPs} \label{chp:pdq}  \index{Q-learning! primal-dual Q-learning}
If we know the  optimal policy of an MDP is monotone, how can Q-learning exploit this structure?  \index{primal-dual algorithm! Q-learning}
If the Q-factors of
a MDP are submodular, then
the optimal policy has a monotone structure. Suppose we do not have explicit knowledge of the transition matrix, but we know that the costs and transition
matrix satisfy sufficient conditions (see \cite{Kri16}) so that  the Q-factors  are  submodular.
 {\em How can the Q-learning algorithm be designed to exploit this submodular property of the Q-factors?}

The submodularity condition is a linear inequality constraint on $Q$:  \index{submodularity! primal-dual Q-learning}
\beq  Q(i,\action+1) - Q(i,\action)  \leq Q(i+1,\action+1) - Q(i+,\action)  \label{eq:submodcons}\eeq
Write this as the inequality constraint 
\beq Q M \geq 0   \label{eq:submodcons2}\eeq where  $\geq$ is element-wise, and the definition of $M$ is obvious from (\ref{eq:submodcons}).

In order to incorporate the constraint (\ref{eq:submodcons2}), it is convenient to  interpret Q-learning as a stochastic gradient algorithm that minimizes
an objective function.
Accordingly, define $g(Q) $ so that $\nabla_Q g(Q) = -f(Q) $ where $f(\cdot)$ is defined in (\ref{eq:fdefq}).
Then we can write  (\ref{eq:rmq}) as 
$  \E\{ f(Q) \} = \nabla_Q \E\{ g(Q) \} = 0$. Then Q-learning can be interpreted as a stochastic approximation algorithm 
to find $$ Q^* = \argmin_Q  \E\{ g(Q) \} $$
From (\ref{eq:ql1}) and (\ref{eq:fdefq}),  $\nabla^2_Q  f(Q)  = -\nabla_Q g(Q) $ is a diagonal matrix with non-negative  elements and 
hence  positive semidefinite.
Therefore $g(Q)$ is convex.

As the objective is convex and the constraint set (linear inequality) is convex, we can use the primal-dual stochastic approximation
algorithm to estimate $Q^*$:
\begin{align*}
\hQ_{k+1} &=\hQ_k + \epsilon_k^{(1)} \left[  f(\hQ_k) +  \lambda_k M \right] \\
\lambda_{k+1} &= \max[  \lambda_k -\epsilon_k^{(1)}Q_k M,\; 0].
\end{align*}
Here $\lambda_k \geq 0$ are interpreted as Lagrange multipliers for the constraint (\ref{eq:submodcons}).
The step sizes $\epsilon_k^{(1)}$ and $\epsilon_k^{(1)}$ are evaluated as in (\ref{eq:qlearnss}). For numerical examples see \cite{DK07}.

\section{Policy  Gradient  Reinforcement Learning  for MDP} \label{sec:policygrad}  \index{policy gradient algorithm! MDP|(}
The Q-learning algorithms described in \secn \ref{sec:qlearning}  operate in the value space and aim to estimate the value function.
The rest of this chapter focuses on solving MDPs and POMDPs using reinforcement learning algorithms that operate in the {\em policy} space. 
That is, with
$\policy_\model$ denoting a policy parametrized by $\model$, the aim is to minimize the expected cumulative  cost
$\E\{\Cost_\bindex(\policy_{\model})\}$ with respect to $\model$ by using a 
stochastic gradient
algorithm of the form\footnote{We use $n$ to denote the batch time index. The policy gradient algorithm  operates on batches of data
where each batch comprises of $\bsize$ time points.}
 \beq  \th_{\bindex+1} = \th_\bindex - \stepsize_\bindex \hat\nabla_\th \Cost_\bindex(\policy_{\model_\bindex}) . \label{eq:sgrad01} \eeq
Here $\Cost_\bindex(\policy_{\model_\bindex})$ denotes the observed cumulative cost by the decision maker when using policy $\policy_{\model_\bindex}$ and 
$ \hat\nabla_\th \Cost_\bindex(\policy_{\model_\bindex}) $ denotes the estimated gradient of the cost $\E\{\Cost_\bindex(\policy_{\model})\}$ evaluated
at $\policy_{\model_\bindex}$.
The phrase ``policy gradient algorithm'' applies to algorithm (\ref{eq:sgrad01}) since it moves along the gradient of the cost in parametrized policy space
$\th$ to determine
the optimal parametrized policy. 

One way of implementing the policy gradient algorithm (\ref{eq:sgrad01}) is to use the finite difference SPSA Algorithm \ref{alg:spsabasic}.  In this section we focus on using the more
sophisticated score function and weak derivative
 gradient estimators of Chapter \ref{chp:simulation}  
to design  policy gradient algorithms
for solving MDPs and constrained MDPs.

Consider an average cost unichain  MDP \cite{Put94}  where $\state_k \in \statespace=\{1,2,\ldots,\statedim\}$ is the state and $\action_k\in \actionspace = \{1,2,\ldots,\actiondim\}$ is the action.
Then  the Markov process  $$z_k= (\state_k,\action_k)$$ has transition matrix 
given by
\beq
\begin{split}  \tp_{i,\action,j,\baction}(\th)
&=  \prob(\state_{k+1}=j,\action_{k+1}=\baction \mid \state_k=i,\action_k=\action) = \th_{j \baction}\,
\tp_{ij}(\action)\\
\text{ where }& \quad  \th_{j\baction}= \prob(\action_{k+1} = \baction| \state_{k+1} = j).  
\end{split}   \label{eq:tpfull} \eeq
The action probabilities $\th$  defined in (\ref{eq:tpfull}) specify the policy for the MDP. 
Let $\beliefm(i,a)$ denote  the stationary distribution of the  Markov process $z_k$. 
Then  we can solve for $\beliefm(i,a)$ as a linear program   and then computed the action probabilities as
$ \th_{ia} = \beliefm(i,a)/ \sum_{i=1}^\statedim \beliefm(i,a)$.
%

Unfortunately, if the transition probabilities $ \tp_{ij}(\action)$ are not known, then the LP      cannot be solved.
Instead, here we consider the following  equivalent formulation  for the optimal policy $\th^*$:
\beq \begin{split}   \th^* &=  \argmin_{\th \in \Model} \Cost(\th),\quad
  \Cost(\th) =  \E_{\beliefm} \{ \cost(\state,\action)\} = \sum_{i \in \statespace} \sum_{\action\in \actionspace} \beliefm(i,a) \cost(i,a) . \\
\Th &= \{\th_{ia} \geq 0, \sum_{a\in \actionspace} \th_{ia} =1 \}.
\end{split}
\label{eq:mlobj}
\eeq
Note from (\ref{eq:mlobj}) that  the optimal policy specified by $\th^*$ depends  on the stationary distribution $\belief_\th$
rather than the unknown transition probabilities $\tp(\action)$.

\subsection{Parametrizations of  Policy} \label{chp:policyparam}
To estimate the optimal $\th^*$ in (\ref{eq:mlobj}),
we need to ensure that $\th^* \in \Th$.
To this end, it is convenient to 
parametrize the action probabilities
 $\th$  by some judiciously chosen parameter vector $\psi$ so that 
\begin{align}
      \th_{ia}(\psi) =  \prob(u_n = a  | \state_n=i) , \quad a \in \actionspace = \{1,2\ldots,\actiondim\},\;i \in \statespace = \{1,2,\ldots,\statedim\}
  \nn
\end{align}
The  optimization problem is then
\beq
\min_{\psi \in \Psi} C(\psi) , \quad \text{ where } 
C(\psi) = \E_{\belief_{\psi}}\{c(x,u)\} 
\eeq
Note that the instantaneous cost is independent of $\psi$ but the expectation is with respect to a measure parametrized by $\psi$.

With the above parametrization,
we will  use the   stochastic gradient algorithm
\beq  \psi_{\bindex+1} = \psi_\bindex - \stepsize_\bindex \hat\nabla_\psi \Cost_\bindex(\psi_\bindex)  \label{eq:sgrad2} \eeq
 to estimate the minimum  $\th(\psi^*)$. Here $\hat\nabla_\psi \Cost_\bindex(\model(\psi_\bindex))$ denotes an estimate
 of the gradient $\nabla_\psi \Cost(\model(\psi))$ evaluated at $\psi_\bindex$.
 The aim is obtain a gradient estimator which does not require explicitly knowing the transition matrices $\tp(\action)$ of the MDP.
The algorithm (\ref{eq:sgrad2})  needs to
operate recursively on batches of the observed system trajectory
$\{(\state_k,\action_k), k \in \{\bindex \bsize, (\bindex+1)\bsize-1\}\}$
to yield
a sequence of estimates $\{\psi_\bindex\}$ of the optimal solution $\psi^*$.

Before proceeding with  algorithm (\ref{eq:sgrad2}) and the gradient estimator,
we first introduce 
 two useful parametrizations $\psi$ that automatically encode the constraints (\ref{eq:mlobj}) on the action probabilities 
$\th$. 

\noindent  {\em 1. Exponential Parametrization}:
The  exponential parameterization for $\th$ is
\beq
\th_{ia}(\psi) = {e^{\psi_{ia}}\over\sum_{u\in\actionspace} e^{\psi_{iu}}},
\quad \psi_{ia} \in \reals, i \in \statespace, \; a \in \actionspace =  \{1,2,\ldots, \actiondim\}. \label{eq:expoparam}
\eeq
In (\ref{eq:expoparam}), the $\psi_{ia} \in \reals$ are unconstrained, yet clearly the  constraint (\ref{eq:mlobj})  holds.
 Note $\psi$ has dimension $\actiondim \statedim $.

\noindent  {\em 2. Spherical Coordinate Parametrization}: To each value
$\th_{iu}$  associate the values $\lambda_{iu} =
\sqrt{\th_{iu}}$. Then (\ref{eq:mlobj}) yields 
$\sum_{u \in \actionspace} \lambda^2_{iu} =1$, and
$\lambda_{iu}$ can be interpreted as the coordinates of a vector that lies
on the surface of the unit sphere in $\reals^\actiondim$.
  In spherical coordinates, the angles are
$\a_{ia}, a =1,\ldots \actiondim-1$, and the radius is 
unity. For  $\actiondim \geq 2$, the spherical coordinates
parameterization $\a$ is defined as:
\begin{equation} \label{eq:alfadef}
\th_{iu}(\a) = 
\begin{cases} \cos^2(\a_{i,1}) & \text{if } u =1\\
              \cos^2(\a_{i,u}) \prod_{p=1}^{u-1} \sin^2(\a_{i,p}) &
2 \leq u\leq \actiondim -1 \\
\sin^2(\a_{i,\actiondim-1)}) \prod_{p=1}^{\actiondim-2} \sin^2(\a_{i,p})  &u = \actiondim
\end{cases}.
\end{equation}
 To summarize, the spherical coordinate  parameters are  $$\a_{ia} \in \reals, \quad i=1,2,\ldots,\statedim,\; a=1,2,\ldots,\actiondim-1. $$
 The $\psi_{ia}$  are unconstrained, yet the constraint (\ref{eq:mlobj}) holds.
 
 Note $\psi$ has dimension $(\actiondim-1) \statedim$.
 For example, if $\actiondim=2$, $\th_{i1} = \cos^2 \a_{i,1}$ and $\th_{i2} = \sin^2 \a_{i,1}$ where $\a_{i1}$ is 
 unconstrained; clearly
$\th^2_{i1} + \th^2_{i2} = 1$.  \index{policy gradient algorithm! MDP|)}

\section{Score Function Policy Gradient Algorithm  for MDP}  
\label{sec:sfmdp}
 \index{gradient estimator! score function}   \index{policy gradient algorithm! score function|(}
This section
uses the score function gradient estimator of Algorithm  \vref{alg:score}  to estimate
  $\hat\nabla_\psi\Cost_\bindex(\psi_\bindex)$.
Together with the stochastic gradient algorithm (\ref{eq:sgrad2}), it constitutes a reinforcement learning algorithm for estimating
the optimal policy $\psi^*$ for the MDP without 
requiring knowledge of the  transition matrices.

We now describe the score function estimator for an MDP.
Consider  the augmented Markov process $z_k=(\state_k,\action_k)$. 
From the  transition probabilities
 (\ref{eq:tpfull}), it follows that
\beq \nabla_\psi  \tp_{i,\action,x,\baction}(\th(\psi)) =\tp_{ix}(\action) \, \nabla_\psi  \th_{x\baction}(\psi) .   
\eeq
The aim is to estimate the gradient  with respect to each component $\psi_{xa}$ for the exponential and spherical parametrizations defined above.

For  the exponential parametrization (\ref{eq:expoparam}),  $\nabla_{\psi_{xa}} \th_{xa}(\psi) = 
\th_{xa} - \th_{xa}^2$ and  $\nabla_{\psi_{xa}} \th_{x\baction}(\psi) = 
 - \th_{xa} \th_{x\baction}$ for $\baction \neq a$. So for $a=1,2,\ldots,\actiondim$,
Step 2 of  Algorithm  \ref{alg:score} for the $\bindex$-th batch comprising of times $k \in \{\bindex\bsize+1,\ldots, (\bindex+1)\bsize\}$ is
\begin{align} S_k^{\psi_{xa}} &= \frac{ {\nabla_{\psi}}_{xa} \tp_{z_{k-1},z_k}(\th)}
{  \tp_{z_{k-1}z_k}(\th)}  + S_{k-1}^{\psi_{xa}}  \label{eq:score1}  \\  
\text{ where } & \frac{ {\nabla_\psi}_{x,a} \tp_{z_{k-1},z_k}(\th)}
{  \tp_{z_{k-1}z_k}(\th)} = \begin{cases} 1 - \model_{\state_k,\action_k}  & \text{ if } a = u_k, x = x_k \\
                     						- \th_{\state_k,a}  & \text{ if } a \in \actionspace - \{u_k\},  x = x_k \\
								0 & \text{ otherwise. }
\end{cases}
\nn
\end{align}
If instead we use 
the spherical coordinates  (\ref{eq:alfadef}), then Step 2 of  Algorithm  \ref{alg:score}  for $a=1,2,\ldots, \actiondim-1$ is:
\begin{align} S_k^{\psi_{xa}} &= \frac{ {\nabla_{\psi}}_{xa} \tp_{z_{k-1},z_k}(\th)}
{  \tp_{z_{k-1}z_k}(\th)}  + S_{k-1}^{\psi_{xa}}  \label{eq:scorespherical}  \\  
  \frac{ {\nabla_\a}_{x a} \tp_{z_{k-1},z_k}(\th)}
{  \tp_{z_{k-1}z_k}(\th)} &=  \begin{cases} \frac{2}{\tan \a_{\state a}} &  a < u_{k} , \state = \state_k \\
							- 2 \tan \a_{\state a}  &  a = u_k, \state = \state_k \\
							0   & a  > u_{k+1}
\end{cases} \nonumber
\end{align}
Finally,
for either parametrization,
Step 3 of Algorithm \ref{alg:score} for the $n$-th batch reads: $\widehat{\nabla}_{\psi_{xa}} C_n(\model) = \frac{1}{\bsize} \sum_{k=1}^\bsize \cost (\state_k,\action_k) S^{\psi_{xa}}_k$.

The stochastic gradient algorithm (\ref{eq:sgrad2}) together with score $S_k^\psi$ constitute a parameter free reinforcement learning
algorithm for solving a MDP. As can be seen from (\ref{eq:score1}),  explicit knowledge of the 
transition probabilities $\tp(\action)$ or costs $\cost(\state,\action)$ are not required; all that is required is that the cost and next state
can be obtained (simulated)  given the current state by action.

 
 The main issue with the  score function gradient estimator is its  large variance  as  described in Theorem \ref{thm:scoremarkov}.
To reduce the variance,  \cite{BB02} replaces  Step~2 with
\beq S^\psi_k =   \frac{ \nabla_\psi \tp_{\state_{k-1}\state_{k}}(\th)}
{  \tp_{\state_{k-1}\state_{k}}(\th) } + \beta\, S^\psi_{k-1} 
\label{eq:discountscore}
\eeq where $\beta\in (0,1)$ is a forgetting factor.
Other variance reduction techniques
 \cite{Pfl96} include
regenerative estimation, finite horizon approximations.
 
\section{Weak Derivative Gradient Estimator for MDP}  \label{sec:wdmdp}
 \index{policy gradient algorithm! weak derivative |(}
This section uses  the weak derivative gradient estimator of  Algorithm  \vref{alg:weak}  to estimate
 $\hat\nabla_\psi\Cost_k(\model_k)$. 
Together with the stochastic gradient algorithm (\ref{eq:sgrad2}), it constitutes a reinforcement learning algorithm 
for estimating the optimal policy $\psi^*$ for an MDP.

Consider  the augmented Markov process $z_k=(\state_k,\action_k)$. 
From the  transition probabilities
 (\ref{eq:tpfull}), it follows that
\beq \nabla_\psi  \tp_{i,\action,x,\baction}(\th) =\tp_{ix}(\action) \, \nabla_\psi  \th_{x,\baction} . \label{eq:wdfact} \eeq

Our plan is as follows:
Recall  from Algorithm \vref{alg:weak}  that
 the weak derivative estimator generates two Markov process $\dot{z}$ and $\ddot{z}$. The weak derivative  representations we  will choose below 
 imply that 
 $\dot{z}_k = z_k$ for all time $k$, where $z_k$ is the sample path of the MDP.
 So we only need to worry about simulating the process $\ddot{z}$. It is shown later in this section
that  $\ddot{z}$ can also be obtained from  the MDP sample path $z$ by using  cut-and-paste arguments.

For the exponential parameterization (\ref{eq:expoparam}), one obtains from (\ref{eq:wdfact}) the following derivative with respect
to each component $\psi_{xa}$:
$$  \nabla_{\psi_{xa}} P(i,u,x,\baction)  =  \th_{xa} ( 1- \th_{xa}) \left[ \tp_{ix}(u) \,I(\baction = a)  -  \tp_{ix}(u)  \,\frac{ \th_{x\baction}}{1 - \th_{xa}} 
\, I(\baction \neq a)
\right]. $$

For the spherical coordinate parameterization (\ref{eq:alfadef}), 
elementary calculus yields  that with respect to each component
$\a_{xa}$, $a=1,2,\ldots, \actiondim-1$
$$  \nabla_{\a_{xa}} P(i,u,x,\baction)  =  -2  \th_{xa} \, \tan \a_{xa}
\left[ \tp_{ix}(u) \,I(\baction = a) -  \tp_{ix}(u) \frac{\th_{x\baction}}{\th_{xa} \, \tan^2 \a_{xa}}  I(\baction > a)\right]. $$
Comparing this with the weak derivative decomposition (\ref{eq:wdmarkov}),  
we can use  the weak derivative 
gradient estimator of Algorithm \ref{alg:weak}.   This yields the Algorithm \ref{alg:wdmdp} which is the weak derivative estimator an MDP.

\begin{algorithm}
\caption{Weak Derivative Estimator for MDP} \label{alg:wdmdp}
Let $k=0,1,\ldots,$ denote local time within the $n$-th batch.
\begin{compactitem}
\item[Step 1.] Simulate $z_0,\ldots,z_{m-1}$ with transition probability $\tp_{i,\action,j,\baction}(\th(\psi))$ defined in (\ref{eq:tpfull}).
\item[Step 2a.] Starting with $z_{m-1}$, choose $\dot{z}_m = z_m = (\state_m,\action_m)$. Choose $\ddot{\state}_m = \state_m$.\\
Choose 
 $\ddot{z}_m = (\ddot{\state}_m,\ddot{\action}_m)$ where
$\ddot{\action}_{m}$ is simulated with 
\begin{align*}
\prob( \ddot{\action}_{m} = \baction )  &=  \frac{ \th_{\state_{m}\baction}}{1 -  \th_{\state_{m}a}},    \quad \baction \in \actionspace - \{a\}   \text{ (exponential parameterization)} \\
\prob( \ddot{\action}_{m} = \baction )  &=  \frac{\th_{x_m \baction}}{\th_{x_m a} \, \tan^2 \a_{x_m a}},  \quad  \baction \in \{a+1,\ldots,\actiondim\}
  \text{ spherical coordinates)} \end{align*}
\item [Step 2b.] 
 Starting with $\dot{z}_{m}$ and $\ddot{z}_m$,
 simulate the two Markov chains $\dot{z}_{m+1}, \dot{z}_{m+2},\ldots $  and  $\ddot{z}_{m+1}, \ddot{z}_{m+2},\ldots $ with transition 
 probabilities as in Step 1.
(Note $\dot{z}_k = z_k$ for all time $k$ by construction.)
\item[Step 3.]  Evaluate the weak derivative estimate for the $n$-th batch  as 
\beq  \widehat{\nabla}_{\psi_{xa}}  C_n(\model)  = \wderiv_{x_m a} 
\sum_{k=m}^{m+\tau} \cost(z_k) -
 c(\ddot{z}_k), \text{ where } \tau = \min\{k: {z}_k = \ddot{z}_k\}, \label{eq:wdcostexpo}
\eeq
$$ \wderiv_{x a} = \begin{cases}  \th_{xa} ( 1- \th_{xa}) &  \text{ exponential parametrization }  \\
-2  \th_{xa} \, \tan \a_{xa} & \text{ spherical coordinates } 
\end{cases}
$$
\end{compactitem}
\end{algorithm}

Finally, the stochastic gradient algorithm (\ref{eq:sgrad2}) together with Algorithm \ref{alg:wdmdp}
result in a policy gradient algorithm for solving an MDP.

\subsection*{Parameter Free Weak Derivative Estimator}
Unlike the score function gradient estimator (\ref{eq:score1}), the weak derivative estimator   requires  explicit knowledge of the 
transition probabilities $\tp(\action)$ in order to propagate the process $\{\ddot{z}_k\}$ in the evaluation of
(\ref{eq:wdcostexpo}) in Algorithm \ref{alg:wdmdp}.

How can the weak derivative  estimator in  Algorithm \ref{alg:wdmdp} be modified to work without knowledge of the transition matrix? This is described in 
\cite{Kri16,KV12}
We need to propagate $\ddot{z}_k$ in Step 2b.
This is done  by a cut and paste technique originally proposed by \cite{HC91}.
Given $\ddot{z}_m = (\state_m, \ddot{u}_m)$ at time $m$, define $$\nu = \min\{k > 0: z_{m+k} = ( \state_m, \ddot{\action}_m)\}. $$  
Since $z$ is unichain, it follows that  $\nu$ is finite with probability one.
Then $\ddot{z}$ is constructed as follows: 
\\
{\em Step (i)}:
Choose $\ddot{z}_{m+k} = z_{m + \nu+k}$ for $k=1,2,\ldots, \batchsize$ where $\batchsize$ denotes some pre-specified batch size.
\\ {\em Step (ii)}: Compute  the cost differences in (\ref{eq:wdcostexpo}) as
$$\sum_{k=m}^{\batchsize-\nu} \cost(z_k) - \cost(\ddot{z}_k)  = \sum_{k=m}^{m+\nu-1} c(z_k)  + \cancel{\sum_{k=m+\nu}^{\batchsize-\nu} c(z_{k}) }
- \cancel{\sum_{k=m+\nu}^{\batchsize - \nu} c(z_k)}
 - \sum_{k=\batchsize-\nu + 1}^\batchsize c(z_k)   $$
Another possible implementation is
\beq \sum_{k=m}^\batchsize c(z_k) - \sum_{k=m+\nu}^{\batchsize} c(\ddot{z}_k) - \nu\, \hat{c} =
\sum_{k=m}^{\batchsize - \nu-1} c(z_k)   - \nu \, \hat{c}
\label{eq:wdcostest} \eeq
Here $ \hat{c} $ denotes any estimator that converges
as $\batchsize \rightarrow \infty$  to $C(\a)$ a.s., where $C(\a)$ is defined in (\ref{eq:mlobj}).  For example,
$ \hat{c} = \frac{1}{\batchsize} \sum_{m=1}^\batchsize  c (z_m)$.
Either of the above implementations in (\ref{eq:wdcostexpo}) together with Algorithm \ref{alg:wdmdp} results in
a policy gradient algorithm (\ref{eq:sgrad2}) for solving the MDP without explicit knowledge of the transition matrices.
 \index{policy gradient algorithm! weak derivative |)}

\section{Numerical Comparison of  Gradient Estimators}
\label{sec:numericalwd}
 This section compares the  score function gradient estimator (used in the influential paper  \cite{BB02})
 with the parameter free weak derivative  estimator of \secn \ref{sec:wdmdp}. It is shown that the weak derivative estimator
 has substantially smaller variance.

 The following 
MDP was simulated:  $\statespace=\{1,2\}$ (2 states), $\actiondim= 3$ (3 actions),
\[
\tp(1) = 
\begin{pmatrix}
0.9 & 0.1 \\
0.2 & 0.8 
\end{pmatrix}, \;\;
\tp(2) =
\begin{pmatrix}
0.3 &  0.7\cr
0.6 & 0.4 \cr
\end{pmatrix}, \;\;\tp(3) = 
\begin{pmatrix}
0.5 &  0.5\cr
0.1 & 0.9 \cr
\end{pmatrix}.
\]
The action probability matrix  $(\th(i,\action))$ and cost 
matrix $(c(i,\action))$ were chosen as:
\[  ( \th(i,a)) = \begin{bmatrix}
0.2 &  0.6 &  0.2 \\ 0.4&  0.4&   0.2 \end{bmatrix},\quad
(c(i,a)) = 
-\begin{bmatrix}
50.0 &   200.0 &  10.0 \\ 
3.0 &    500.0 &   0.0 \end{bmatrix}
\]
 We work 
with the exponential parametrization (\ref{eq:expoparam}). First we compute the ground-truth.
By solving the linear program for
the  optimal parameter $\psi^*$, we obtain the true derivative at $\psi^*$
as
\begin{equation} \label{eq:theorval}
\nabla_{\psi}[C(\th(\psi))]
=\begin{pmatrix}
  -9.010  & 18.680 &   -9.670  \\    
-45.947  & 68.323 & -22.377            
\end{pmatrix}. 
\end{equation}
We simulated the parameter free weak derivative estimator using (\ref{eq:wdcostest})  in Algorithm \ref{alg:wdmdp}
for the exponential parametrization.
 For 
batch sizes $N=100$ and $1000$ respectively,  the 
weak derivative  gradient estimates are
\begin{align*}
 \GF_{100} &=
\begin{pmatrix}
-7.851 \pm 0.618    &  17.275 \pm 0.664   &  -9.425 \pm 0.594 \\
-44.586 \pm 1.661 &  66.751 \pm 1.657    &  -22.164\pm 1.732 \\
\end{pmatrix}\\
\GF_{1000} &= 
\begin{pmatrix}
 -8.361 \pm 0.215    &  17.928 \pm 0.240    &  -9.566 \pm 0.211 \\
 -46.164 \pm 0.468   &  68.969 \pm 0.472    &  -22.805 \pm 0.539 \\
\end{pmatrix}.
\end{align*}
The numbers after $\pm$ above, denote the confidence
intervals at level
$0.05$  with $100$ batches.
The variance of the  gradient estimator  is shown in
Table \ref{tab:PhanTheta}, together with the corresponding CPU time. 
\begin{table}[h]
\centering
\begin{tabular}{|c|| c | c | c| } \hline
$N=1000$  &  
  \multicolumn{3}{|c|}{$\var [\GF_N]$} \\
 \hline
$ i = 1 $ & 1.180 &  1.506 & 1.159   \\
 $i = 2 $ &  5.700 & 5.800 &  7.565 \\
\hline
CPU   & \multicolumn{3}{c|}{1 unit} \\
\hline
\end{tabular}
\caption{Variance of weak derivative estimator with  exponential parametrization}
\label{tab:PhanTheta}
\end{table}

We implemented the score function gradient estimator of \cite{BB02}
with the  following
parameters: forgetting
factor $1$,  batch
sizes of $N=1000$ and 10000. In both cases a total number of $10,000$
batches were simulated.
The score function gradient estimates are
\begin{align*}
 \GS_{10000} &=
\begin{pmatrix}
 -3.49  \pm 5.83     &  16.91 \pm 7.17 & -13.42 \pm 5.83 \\
 -41.20 \pm 14.96 &  53.24 \pm 15.0  &  -12.12 \pm 12.24 
\end{pmatrix}\\
 \GS_{1000} &= 
\begin{pmatrix}
-6.73 \pm 1.84 & 19.67 \pm 2.26 &  -12.93 \pm 1.85 \\
-31.49 \pm 4.77 & 46.05 \pm 4.75 &  -14.55 \pm 3.88
\end{pmatrix}
\end{align*}
The variance of the score function gradient estimates are given
Table \ref{tab:PhanScore}. 
\begin{table}[h]
\centering
\begin{tabular}{|c| c | c | c |} \hline
$N=1000$  & \multicolumn{3}{c|}{$\var [\GS_N]$}  \\
 \hline
$ i = 1 $ & 89083 & 135860 & 89500  \\
 $i = 2 $ & 584012 & 593443 & 393015 \\
\hline
CPU & \multicolumn{3}{|c|}{687 units}  \\
\hline
\end{tabular}
\begin{tabular}{|c| c | c | c |} \hline
$N=10000$  & \multicolumn{3}{c|}{$\var [\GS_N]$}  \\
 \hline
$ i = 1 $ & 876523 & 1310900 & 880255  \\
 $i = 2 $ & 5841196 & 5906325 & 3882805 \\
\hline
CPU & \multicolumn{3}{|c|}{6746 units}  \\
\hline
\end{tabular}
\caption{Variance of Score Function estimator}
\label{tab:PhanScore}
\end{table}

Notice that even with substantially larger batch sizes and number of batches
(and hence computational time),
the variance of the score function estimator is orders of magnitude larger
than the weak derivative  estimator.
For further numerical results please see \cite{VK03,VKM03,KV12}.

\section{Policy  Gradient  Reinforcement Learning  for Constrained MDP (CMDP)} \label{sec:cmdpgrad}
 \index{policy gradient algorithm! constrained MDP|(}
This section
 describes how the policy gradient algorithms for MDPs described in \secn \ref{sec:policygrad}-\ref{sec:wdmdp}  can be extended to solve constrained MDPs (CMDPs). Recall that for CMDPs, the optimal policy is randomized.
Assuming a unichain CMDP, our aim is to obtain the optimal parametrized policy:
\begin{align} 
\text{ Compute }  & \th^* = \argmin_\model \Cost(\theta) =  \E_{\beliefm} \{ \cost(\state,\action)\} = \sum_{i \in \statespace} \sum_{\action\in \actionspace} \beliefm(i,\action) \cost(i,\action) 
\label{eq:cobj} \\
\text{ subject to } &   \G_l(\th) = 
 \sum_{i} \sum_{\action } \beliefm(i,\action) \, \con_l(i,\action) \leq \rcon_l, \quad l = 1,2,\ldots, L. \label{eq:conmdp2}
\end{align}
The key difference  compared to the unconstrained MDP are the $L$ constraints (\ref{eq:conmdp2}).
As in the unconstrained MDP case (previous section),
 $\beliefm(i,a)$ is the stationary distribution of the controlled Markov process $z_k= (\state_k,\action_k)$.
  The optimal policy
depends  on
the action probabilities  $\th(\psi^*)$ where
$ \th_{ia}(\psi) =  \prob(u_n = a  | \state_n=i)$ and $\psi$ denotes a suitably chosen parametrization of $\model(\psi)$.

The optimal policy of a CMDP is randomized for up to $L$ states.
 Our  aim is to devise
stochastic approximation algorithms to estimate this randomized policy, i.e.,  optimize
(\ref{eq:cobj}) subject to (\ref{eq:conmdp2}). 
 
\noindent {\bf Assumption (O)}: 
The minima
 $\psi^*$ of (\ref{eq:cobj}), (\ref{eq:conmdp2})
 are regular, i.e., $\nabla_\psi \G_l(\psi^*)$, $l=1,\ldots,L$ are linearly independent.
Then $\psi^*$ belongs to the set of Kuhn Tucker points
$$
 \biggl\{ \psi^* \in \Psi:  \;\exists\, \mu_l \geq 0, \,
l=1,\ldots, L\;\text{ such that
} \quad
  \nabla_\psi C + \nabla_\psi \G \mu = 0, \quad
 \G^{\p} \mu = 0 \biggr\}
$$
where $\mu = (\mu_1\ldots,\mu_L)^{\p}$.
Also,  $\psi^*$ satisfies the second order sufficiency condition
$\nabla_\psi^2 C(\psi^*) + \nabla_\psi^2 \G(\psi^*) \mu > 0$ (positive definite)
on the subspace 
$\{ y \in \reals^L: \nabla_\psi \G_l(\psi^*) y=0\}$ for all $l: \G_l(\psi^*)=0,
\mu_l > 0$.

\begin{algorithm}[h]
\caption{Primal-Dual Reinforcement Learning Algorithm for CMDP} \label{alg:pdreinf}

{\em Parameters}: Cost matrix $(c(i,a))$, constraint matrix $(\con(i,a))$,
batch size $N$,  step size $\epsilon>0$

{\em Step 0. Initialize}: Set $n=0$, initialize $\a(n) $ and 
vector $\lambda(n) \in \reals_+^L$.

{\em Step 1. System Trajectory Observation}: Observe  MDP over batch
$ I_n \ole  \{k \in [nN, (n+1)N-1]\}$ using randomized policy $\th(\a(n))$ of
(\ref{eq:alfadef}) and
compute estimate $\hat{\G}(n)$ of the constraints as
\beq
\hat{\G}_l^{\epsilon}(n+1) =
\hat{\G}_l^{\epsilon}(n) + \sqrt{\epsilon}
\left(  \frac{1}{N} \sum_{k\in I_n} \con_l(Z_k)   - \hat{\G}_l^{\epsilon}(n)\right), \quad l=1,\ldots, L.
\label{eq:spd3}
\eeq

{\em Step 2. Gradient Estimation}: Compute
 $\widehat{\nabla_\psi C}(n)$,  $\widehat{\nabla_\psi \G}(n)$ over the batch
$I_n$ using a gradient estimator (such as weak derivative or score function).

{\em Step 3. Update Policy $\th(\a(n))$}:
Use a penalty function primal dual based stochastic approximation
algorithm  to update $\a$ as follows:
\begin{align}
 \a_{n+1} &= \;
\a_n - \epsilon \biggl(\widehat{\nabla_\psi C}(n) +
\widehat{\nabla_\psi \G}(n) \max \left[0,\lambda_n + \penaltyweight\, \widehat{\G}(n)
\right]\biggr) 
\label{eq:sam1}\\
\lambda_{n+1} &= \max\left[\left(1-\frac{\epsilon}{\penaltyweight}\right) \lambda_n ,\lambda_n + \epsilon \widehat{\G}_n
\right]. \label{eq:sam2} \end{align}
The ``penalization'' $\penaltyweight$ is a suitably large positive constant and
$\max[\cdot,\cdot]$ above is taken element wise.

{\em Step 4}. Set $n=n+1$ and go to Step 1. \index{primal-dual algorithm! policy gradient reinforcement learning}
\end{algorithm}

The policy gradient algorithm  for solving the CMDP is described in Algorithm~\ref{alg:pdreinf}.
Below we elaborate on the primal-dual algorithm (\ref{eq:sam1}), (\ref{eq:sam2}).
The constrained optimization problem (\ref{eq:cobj}), (\ref{eq:conmdp2}) is in general non convex in the parameter $\psi$. One solution methodology is to minimize the augmented Lagrangian via a 
primal-dual stochastic gradient algorithm.  
Such multiplier algorithms
are widely used in deterministic optimization \cite[pg 446]{Ber00}  with extension to
stochastic approximation in \cite{KC78}.
First, convert the inequality MDP constraints (\ref{eq:conmdp2}) to equality constraints
by introducing the variables $z = (z_1,\ldots,z_L) \in \reals^L$, so that
$ \G_l(\a) + z_l^2  = 0$, $l=1,\ldots,L$. 
Define the augmented Lagrangian,
\begin{equation} \La_\penaltyweight(\a,z,\lambda) \ole
 C(\a)+ \sum_{l=1}^L \lambda_l (\G_l(\a) + z_l^2)
 + \frac{\penaltyweight}{2} \sum_{l=1}^L \left(\G_l(\a) +z_l^2 \right)^2 .
\label{eq:optimp} \end{equation}
Here  $\penaltyweight$ denotes a large positive constant.
After some further calculations detailed in \cite[pg.396 and 397]{Ber00},
 the primal-dual algorithm operating on the augmented Lagrangian reads
\begin{align}
 \a^{\epsilon}_{n+1} &=\a^{\epsilon}_n - \epsilon \biggl({\nabla_\psi C}(\a^{\epsilon}_n) +
{\nabla_\psi \G}(\a^{\epsilon}_n) \max\biggl[0,\lambda^{\epsilon}_n + \penaltyweight \, {\G}(\a^{\epsilon}_n)
\biggr]\biggr) \nn \\
\lambda^{\epsilon}_{n+1} &= \max\left[\left(1-\frac{\epsilon}{\penaltyweight}\right)\lambda^{\epsilon}_n,\lambda^{\epsilon}_n + \epsilon {\G}(\a^{\epsilon}_n)
\right]\label{eq:pd1} \end{align}
where $\epsilon>0$ denotes the step size and 
$\max[\cdot,\cdot]$ is taken elementwise.


\begin{lemma} \label{lem:pdl}
Under Assumption (O),
for sufficiently large $\penaltyweight>0$,
there exists $\bar{\epsilon}>0$, such that for all $\epsilon \in (0,\bar{\epsilon}]$,
the sequence $\{\a^{\epsilon}(n),\lambda^{\epsilon}(n)\}$ generated by the primal-dual algorithm (\ref{eq:pd1})
is attracted to a local Kuhn-Tucker pair $(\a^*, \lambda^*)$.
\end{lemma}
\begin{proof}  The proof 
follows from 
 Proposition 4.4.2 in \cite{Ber00}.
\end{proof}

 \index{policy gradient algorithm! constrained MDP|)}

In \cite{Kri16} we extend the above algorithms to obtain policy gradient algorithms for POMDPs.

\section{\pwe}
The book \cite{SB98} is an influential book on  reinforcement learning. 
\cite{BT96} (though published in the mid 1990s) is a remarkably clear exposition on reinforcement learning algorithms. 
\cite{KY03,BT96} present convergence analyses proofs for several types of reinforcement learning algorithms.
\cite{BY12} has several  novel variations of the Q-learning algorithm for discounted cost MDPs.
\cite{ABB01} has  novel Q-learning type algorithms for average cost MDPs. 
\cite{DG14} uses ant-colony optimization for Q-learning.
 \cite{BB02} is an influential paper that uses the score function
gradient estimator in a  policy gradient algorithm for POMDPs. 

The  area of reinforcement learning is evolving with rapid dynamics. The proceedings
of the  Neural Information Processing Systems (NIPS) conference and International Conference on Machine Learning (ICML)
have numerous recent advances in reinforcement learning algorithms.


\chapter{Stochastic Approximation Algorithms: Examples} 
\label{chp:markovtracksa}

\minitoc
This  chapter,
 presents case studies of stochastic approximation algorithms in state/parameter estimation and modeling in 
the context of POMDPs.
%

Example 1 discusses online estimation of  the parameters of an HMM using the recursive maximum likelihood estimation algorithm.
The motivation stems from classical adaptive control: the parameter estimation algorithm can be used to estimate the parameters of the POMDP
for a fixed policy.

Example 2 shows that for an HMM  comprised of a slow Markov chain, the least mean squares algorithm  can provide satisfactory state estimates
of the Markov chain
without any knowledge of the underlying parameters. In the context of POMDPs, once the state estimates are known, a variety of suboptimal 
algorithms can be used to synthesize  a reasonable policy.

Example 3 shows how discrete stochastic optimization problems can be solved via stochastic approximation algorithms.
In controlled sensing, such algorithms
can be used to compute the optimal sensing strategy from a finite set of policies.

Example 4 shows how large scale Markov chains can be approximated by a system of ordinary differential equations. This mean field analysis is illustrated in the context
of information diffusion in a social network.



\section{Convergence of Stochastic Approximation Algorithms}  \index{stochastic approximation algorithm} \label{sec:primersa}
This section presents a rapid summary of the
convergence analysis of stochastic approximation algorithms.
The books \cite{BMP90,SK95b,KY03} are seminal works that study the convergence of stochastic approximation algorithms under general conditions.

Consider a constant step size stochastic approximation algorithms  of the form
\beq  \theta_{k+1} = \theta_k + \stepsize \,H(\theta_k, \state_k), \quad k=0,1,\ldots  \label{eq:sabasic}\eeq
where $\{\theta_k\}$ is a sequence of parameter estimates generated by the algorithm, $\stepsize$ is small positive fixed  step size,
and $\state_k$ is a discrete time geometrically ergodic Markov process (continuous or discrete state) with transition kernel 
$\tp({\theta_k})$ 
and stationary distribution $\belief_{\theta_k}$.
%
%

Such algorithms are useful for tracking time varying parameters and are widely studied in adaptive filtering.
Because of the constant step size, convergence with probability one of the sequence $\{\theta_k\}$ is ruled out. Instead, under reasonable conditions,  $\{\theta_k\}$
converges weakly (in distribution) as will be formalized below.

Analysis of stochastic approximation algorithms is typically of three types:
\begin{compactenum}
\item Mean Square Error Analysis
\item  Ordinary Differential Equation (ODE) Analysis  \index{ordinary differential equation (ODE) analysis}
\item Diffusion Limit for Tracking Error
\end{compactenum}

The mean square analysis seeks to show that for large time $k$, $\E\|\th_k - \th^* \|^2 = O(\e)$ where $\th^*$ is the true parameter.  \secn \ref{sec:trackingerror}
provides a detailed example.

 In comparison, the ODE analysis and diffusion limit deal with suitable scaled sequences of iterates that are treated
as {\em stochastic processes} rather than random variables. These two analysis methods seek to characterize the behavior of the entire trajectory 
(random process) rather than
just at a specific time $k$ (random variable) that  the mean square error analysis does. The price to pay for the ODE and diffusion limit 
analysis is the highly technical machinery of weak convergence analysis. Below we will provide a heuristic incomplete treatment
that only scratches the surface
of this elegant 
 analysis tool.

\subsection{Weak Convergence}

For constant step size algorithms, under reasonable conditions one can show that  the estimates generated
by the algorithm converge weakly.  For a comprehensive treatment of  weak convergence of Markov processes please see
\cite{EK86}. 

Weak convergence is a function space generalization of convergence in distribution of random variables.\footnote{A sequence of  random variables  $\{X_n\}$  converges in distribution  if
 $ \lim_{n\rightarrow \infty}\cdf_n(x)  =  \cdf(x)$  for all $x$ for which $\cdf$ is continuous. Here $\cdf_n$ and $\cdf$ are the cumulative
 distribution functions of $X_n$ and $X$, respectively.  An equivalent statement is that 
 $\E\{\phi(X_n)\} \rightarrow \E\{\phi(X)\}$ for every
 bounded continuous function $\phi$.}
Consider a  continuous time random process  $X(t), t \in [0,T]$ which we  will denote as $X$.
 A sequence of random processes  $\{X^{(n)}\}$ (indexed by $n=1,2,\ldots$) converges weakly to $X$ if for each bounded continuous real-valued functional~$\phi$,  $$\lim_{n \rightarrow \infty} \E\{ \phi(X^{(n)} )\} =  \E\{ \phi(X)\}.$$ Equivalently, a sequence of probability measures
 $\{P^{(n)}\}$ converges weakly to $P$ if  $\int \phi\,  dP^{(n)} \rightarrow \int \phi \, dP $ as $n \rightarrow \infty$.
Note that the  functional $\phi$ maps the entire trajectory of $X^{(n)}(t), 0 \leq t \leq T$ of the random process to a real number. (The  definition specializes
to the classical convergence in distribution if $X^{(n)}$ is a random variable and $\phi$ is a function mapping $X^{(n)}$ to a real number.)
 
In the above definition, the trajectories of the random processes lie in the function space $C[0,T]$ (the class of continuous functions
 on the interval $[0,T]$), or
more generally,
 $D[0,T]$ which
is the space of piecewise constant functions  that are continuous from the right with limit on the left -  these are called 
 `cadlag' functions (continue \'a droite, limite \'a gauche) in French.


  \index{stochastic approximation algorithm! analysis|(}

\subsection{Ordinary Differential Equation (ODE) Analysis of Stochastic Approximation Algorithms} \index{stochastic approximation algorithm! ODE analysis}     \index{ordinary differential equation (ODE) analysis|(}
Consider a generic stochastic approximation algorithm 
\beq  \theta_{k+1} = \theta_k + \stepsize \,H(\theta_k, \state_k), \quad k=0,1,\ldots,  \label{eq:gensa} \eeq
where $\{\state_k\} $ is a random process
and $\theta_k \in \reals^p$  is the  estimate generated by
the algorithm at time $k = 1,2,\ldots$.

The ODE analysis for stochastic approximation algorithms was pioneered by Ljung (see  \cite{Lju77,LS83}) and subsequently by Kushner and
co-workers \cite{Kus84,KY03}. It aims  to show that the sample path
$\{\th_k\}$ generated by  the stochastic approximation algorithm (\ref{eq:gensa}) behaves asymptotically according to a
 deterministic ordinary differential equation (ODE). 
Let us make this more precise.
 First, we represent the  sequence $\{\th_k\}$ generated by  the
 algorithm (\ref{eq:gensa}) as a  continuous time process since we want to show that its limiting behavior is the continuous time ODE.
 This is done by
constructing the continuous time trajectory via piecewise constant interpolation
of   $\{\th_k\}$  as
\beq  \th^\stepsize(t)= \th_k
\quad \text{ for } t\in [k\stepsize, k\stepsize + \stepsize), \quad  k=0,1,\ldots. \label{eq:interpolate} \eeq
We are interested in studying the limit
 of the continuous time interpolated process  $ \th^\stepsize(t)$ as $\stepsize \rightarrow 0$  over the time interval $[0,T]$.

By using stochastic averaging theory, it can be shown that the following weak convergence result holds; 
see \cite{KY03} for proof.\footnote{For random variables,
convergence in distribution to a constant implies convergence in probability. The generalization to weak convergence is: Weak convergence to a deterministic trajectory (specified in our case by an ODE) implies convergence in probability to this deterministic trajectory. The statement in the theorem is a consequence of this.} 
\begin{theorem}[ODE analysis]   \label{thm:odeweak}
Consider the stochastic approximation algorithm (\ref{eq:gensa}) with constant step size $\epsilon$. Assume
\begin{compactitem}
\item[\rm{\bf (SA1)}] $H(\th,\state)$ is uniformly bounded for all $\th \in \reals^p$ and $\state \in \reals^q$.
\item[\rm{\bf (SA2)}] For any $\ell\ge 0$, there exists $h(\th)$ such that
\beq
    \frac{1}{N} \sum^{N+\ell-1}_{k=\ell}
    \E_l \big\{ H(\th,\state_k) \big\} \to  h(\th)
    \mbox{ in probability as}\ N\to\infty. \label{eq:hdef}
    \end{equation}
where $\E_l$ denotes expectation with respect to the $\sigma$-algebra generated by $\{x_k, k < l\}$.
\item[\rm{\bf (SA3)}] The  ordinary differential equation  (ODE) \beq  \frac{d \theta(t)}{dt}  = h\big(\th(t) \big) , \quad \theta(0) = \th_0.    \label{eq:odesa} \eeq
has a unique solution for every initial condition.
\end{compactitem}
Then the  interpolated estimates $\th^\stepsize(t)$ defined in (\ref{eq:interpolate}) satisfies
$$ \lim_{\stepsize \rightarrow 0} \P\big(\sup_{ 0 \leq t \leq T} |  \th^\stepsize(t) -  \th(t) | \geq \eta \big) = 0 
\quad \text{ for all } T > 0 , \eta > 0 $$
where $\theta(t)$ is the solution of the ODE (\ref{eq:odesa}).
\end{theorem}
Numerous variants of Theorem \ref{thm:odeweak} exist with weaker conditions.
Even though the proof of Theorem \ref{thm:odeweak} is highly technical, the theorem gives remarkable insight.
It says that for sufficiently small step size, the entire interpolated trajectory of estimates generated by the stochastic approximation
algorithm is captured by the trajectory of a deterministic  ODE. Put differently, if the ODE is designed to follow
a specific deterministic trajectory,
then the stochastic approximation algorithm will converge weakly to this trajectory.

Assumption (SA1) is a tightness condition and   can be weakened considerably as shown in \index{tightness}
the books \cite{BMP90,KY03}.  More generally, (SA1) can be replaced by uniform integrability of $H(\th_k,\state_k)$, namely that for some $\alpha > 0$,
$$\sup_{k \leq  T/\epsilon} \E \| H(\theta_k,\state_k)\|^{1+\alpha} < \infty. $$

Assumption (SA2) is a weak law of large numbers. It allows us to work with correlated sequences whose remote past and distant future are asymptotically independent.
Examples include
sequences of i.i.d. random variables with  \index{martingale}
bounded variance,
 martingale difference sequences with finite second moments,  moving average processes driven
 by a martingale difference sequence, mixing sequences in which remote past and distant future are asymptotically independent, and
 certain non-stationary sequences such as  functions of  Markov chains.

{\em Example}: 
Suppose $\{\state_k\}$ is a geometrical ergodic Markov process with transition  probability matrix (kernel)  $\tp({\theta_k})$ 
and stationary distribution $\belief_{\theta_k}$. Then by Theorem \ref{thm:odeweak}, the stochastic approximation converges to the ODE
\begin{align}
 \frac{d \theta(t)}{dt}  &= h(\th(t)) \nonumber\\
 \text{ where } h(\th(t) &= \int_\statespace H(\theta(t),\state) \belief_{\theta(t)}(\state) d\state  = \E_{\belief_\theta(t)}\{ H(\theta(t),x)\}.
 \label{eq:odesamarkov}
 \end{align}

 \index{ordinary differential equation (ODE) analysis|)}

\subsection{Diffusion Limit for Tracking Error}  \index{stochastic approximation algorithm! diffusion limit}
The ODE analysis of Theorem \ref{thm:odeweak} says that the interpolated trajectory of estimates generated by the algorithm converges weakly to the trajectory of the ODE (\ref{eq:odesa}) over a time
interval $[0,T]$.  \index{diffusion limit!  stochastic approximation algorithm}
What is the rate of convergence? This is addressed by the diffusion limit analysis that we now describe.

Define the scaled tracking error  as the continuous time process
\beq  \tth^\stepsize(t)  = \frac{1}{\sqrt{\stepsize}} (\th^\stepsize(t) - \th(t)) \label{eq:differror} \eeq
where $\th(t)$ evolves according to the ODE  (\ref{eq:odesa}) and $\th^\stepsize(t)$ is the interpolated trajectory of the 
stochastic approximation algorithm. 

Define  the following continuous time 
diffusion process:
\beq  \tth(t)  = \int_0^t \nabla_\theta h(\th(s))  \,  \tth(s) \,  ds +  \int_0^t \onoisecovnew^{1/2}(\th(s))\, d\bm(s) \label{eq:itointegralcov} \eeq
where $h(\th(t))$ is defined in the ODE (\ref{eq:odesa}), the covariance matrix 
\beq
\onoisecovnew(\th) = \sum_{n=-\infty}^\infty \cov_\th\left(H(\th,\state_n), H(\th, \state_0) \right), \label{eq:onoisecovnew}
\eeq
and $\bm(t)$ is standard Brownian motion. The last term in (\ref{eq:itointegralcov})  is interpreted as a stochastic 
(It\^{o}) integral; see for example, \cite{KS91}.

Then the main result regarding the  limit of the tracking error is the following:

\begin{theorem}[Diffusion Limit of Tracking Error]
As $\stepsize \rightarrow 0$,  the scaled error process $ \tth^\stepsize$ defined in (\ref{eq:differror}) converges  weakly to the diffusion
process $\tth$ defined in (\ref{eq:itointegralcov}).
\end{theorem}

\subsection{Infinite Horizon Asymptotics and Convergence Rate}   \index{stochastic approximation algorithm! infinite horizon asymptotics}

The ODE and diffusion analyses  outlined above apply to a finite time interval $[0,T]$.
 Often we are interested in the asymptotic behavior as $T\rightarrow \infty$. In particular, we want an  expression for
the asymptotic  rate of convergence  of a stochastic
approximation algorithm. 
In stochastic analysis, the  rate of convergence refers to the asymptotic variance of the normalized errors about the limit point.

Under stability conditions so that 
$\lim_{T \rightarrow \infty}  \lim_{\epsilon \rightarrow 0 }  \th^\stepsize(t) =  \lim_{\epsilon \rightarrow 0 }\lim_{T \rightarrow \infty} \th^\stepsize(t)$,
it can be shown \cite{KY02} that
the stable fixed  points of the ordinary
differential equation (\ref{eq:odesa}) coincide with the attractors of the stochastic approximation algorithm. Suppose that this limit exists and denote
it as $\theta^*$.
Then for large $t$, the diffusion approximation  (\ref{eq:itointegralcov}) becomes the linear (Gaussian) diffusion
\beq  \tth(t)  = \int_0^t \nabla_\theta h(\th)\vert_{\th=\th^*}  \,  \tth(s) \,  ds +  \int_0^t \onoisecovnew^{1/2}(\th^*)\, dw(s). \label{eq:statcov} \eeq
In other words, the scaled error $ \tth^\stepsize(t) $ converges to a Gaussian process -- this can be viewed
as a  functional central limit theorem.

Suppose that the matrix $\nabla_\theta h(\th^*)$ is stable, that is, all its eigenvalues have strictly negative real parts.  Then the
diffusion process  $  \tth(t) $
has a stationary Gaussian distribution, that is for large $t$,
\beq  \tth(t) \sim \normal(0, \kalmancov) \label{eq:gaussapproxsa}\eeq
where the positive definite matrix $\kalmancov$ satisfies the algebraic Lyapunov equation \index{Lyapunov equation! algebraic}
\beq  \nabla_\theta h(\th) \, \kalmancov +\kalmancov \,  \nabla^\p _\theta h(\th) + \onoisecovnew(\th) \big\vert_{\th=\th^*} = 0  .\label{eq:asymptoticrate}\eeq
where $\onoisecovnew$ is defined in (\ref{eq:onoisecovnew}).
\index{stochastic approximation algorithm! asymptotic convergence rate}
The  covariance matrix $\kalmancov$, which is the solution of (\ref{eq:asymptoticrate}), is interpreted as the 
{\em asymptotic rate of convergence} of the stochastic approximation algorithm.  For large $n$
and small step size $\stepsize$, (\ref{eq:gaussapproxsa}) says that the tracking error of the stochastic approximation algorithm (\ref{eq:gensa}) behaves as
$$ \stepsize^{-1/2} (\th_n - \th^*) \sim \normal(0,\kalmancov),$$  \index{stochastic approximation algorithm! asymptotic convergence rate}
where $ \kalmancov$ is the solution of the algebraic Lyapunov equation (\ref{eq:asymptoticrate}).
  \index{stochastic approximation algorithm! analysis|)}

\section{Example 1: Recursive Maximum Likelihood Parameter Estimation of HMMs} \label{sec:rmle} 
\index{recursive maximum likelihood estimation for HMM|(}
\index{HMM parameter estimation! recursive maximum likelihood}

Let $\state_k$, $k=0,1,\ldots$ denote a Markov chain on the state space $\statespace = \{1,2,\ldots,\statedim\}$ where $\statedim$ is fixed
and known.
The HMM parameters $\tp(\th)$ (transition probability matrix) and $\oprob(\th)$ (observation probabilities)
are functions of the parameter
vector $\th$ in a compact subset $\Theta$ of Euclidean space.
Assume that the HMM observations $\obs_k$, $k=1,2,\ldots$ are generated by a true parameter vector 
$\truemodel \in \Theta$ which is not known. The aim is to design a recursive algorithm to estimate the HMM parameter vector $\truemodel$.

{\em Example}: Consider the Markov chain observed in  Markov modulated Gaussian noise via the observation process
$$y_k = \state_k + \sigma(\state_k) \onoise_k $$
where $\onoise_k \sim \normal(0,1)$ is i.i.d.\ and $\sigma(1),\ldots,\sigma(\statedim)$ are fixed positive scalars in the interval $[\underline{\sigma},
\bar{\sigma}]$. 
 One possible parametrization of this Gaussian noise HMM is
$$\th = [\tp_{11},\ldots, \tp_{\statedim\statedim}, \sigma(1),\sigma(2),\ldots, \sigma(\statedim) ]^\p. $$ 
So $\Th = \{\th: \tp_{ij} \geq 0,\;\sum_j \tp_{ij} = 1,\;
\sigma(i) \in [\underline{\sigma},
\bar{\sigma}]\}$.  A more useful parametrization
that automatically encodes these constraints on the transition matrix is to use spherical coordinates or the exponential parametrization discussed 
in \chp  \ref{chp:policyparam}.
\\

The normalized log likelihood  of the HMM parameter $\model$  based on the observations $\obs_{1:n} = (y_{1}, \ldots, y_n)$
is
$ l_n(\th) = \frac{1}{n}
\log p(\obs_{1:n}| \th)$.
It
 can be expressed
as the arithmetic mean of terms involving the observations and the 
HMM prediction filter as follows:
\beq \label{eq:lndef}
 l_n(\th) = \frac{1}{n} \sum_{k=1}^n \log \left[ \one^\p \oprob_{\obs_k}(\th) \belief_{k|k-1}^\th
\right] \end{equation}
where $\oprob_{\obs}(\th) = \diag(\pdf(\obs|\state=1,\th),\ldots, \pdf(\obs|\state=\statedim,\th))$.
Here $ \belief_{k|k-1}^\th$ denotes the HMM predictor assuming the HMM  parameter  is $\th$:
\begin{align*}  \belief_{k|k-1}^\th &= [  \belief_{k|k-1}^\th(1), \ldots,  \belief_{k|k-1}^\th(\statedim) ]^\p, \;
\belief_{k|k-1}^\th(i) = \prob_\th(\state_k = i| \obs_1, \ldots . \obs_{k-1}) .\end{align*}
The HMM predictor is computed via the recursion
\beq \belief^\th_{k+1|k} = \frac{ \tp^\p(\th)\oprob_{\obs_k}(\th) \belief^\th_{k|k-1}}{ \one^\p \oprob_{\obs_k}(\th) \belief_{k|k-1}^\th } .
\label{eq:hmmrpred}
\eeq
Define the incremental score vector as
\beq \score(\belief^\th_{k|k-1},\obs_k,\th) = \nabla_\th \log \left[ \one^\p \oprob_{\obs_k}(\th) \belief_{k|k-1}^\th
\right].  \label{eq:incscore} \eeq

Then the recursive MLE (RMLE) algorithm for online estimation of the HMM parameters is a stochastic approximation algorithm of the form
\beq
 \th_{k+1} = \Pi_\Th\big(\th_k + \epsilon \, \score(\belief^{\th_k}_{k|k-1},\obs_k,\th_k) \big) \label{eq:rmlesa}  \eeq where   $\epsilon$ denotes a positive
 step size,
 and  $\Pi_\Th$ denotes
the projection of the estimate to the set $\Th$. 


\subsection{Computation of Score Vector}
The score vector in the RMLE algorithm (\ref{eq:rmlesa}) can be computed recursively
by
differentiating the terms within the summation in (\ref{eq:lndef}) with 
respect to each element $\th(l)$ of $\theta$, $l=1,\ldots,p$. This yields the $l$-th component of $\score$ as
\beq   
 \score^{(l)}(\belief^{\th}_{k|k-1},\obs_k,\th) =
\frac{\one^\p \oprob_{\obs_k}(\th) \,w_k^{(l)}(\th)} 
{ \one^\p \oprob_{\obs_k}(\th) \belief_{k|k-1}^\th }
+ \frac{ \one^\p [\nabla_{\th(l)}  \oprob_{\obs_k}(\th) ] \belief_{k|k-1}^\th } { \one^\p \oprob_{\obs_k}(\th) \belief_{k|k-1}^\th }
\label{eq:score} \end{equation}
where
%
$w_k^{(l)}(\th) = \nabla_{\th(l)}  \belief_{k|k-1}^{\th}$ denotes the partial
derivative of $\belief^{\th}_{k|k-1}$ in (\ref{eq:hmmrpred}) with respect to the $l$th component of the
 parameter
vector $\th$. 
Define the $\statedim \times p$ matrix $w_k(\th) = ((w^{(1)}_k(\th),
\ldots,w^{(p)}_k(\th))$. Clearly 
 $w_k(\th)$  belongs to $\Xi$ defined by
$\Xi = \{w \in \reals^{\statedim \times p}: {\bf 1}^{\prime} w = \zero\}$ since
$\one^\p \nabla_{\th(l)}  \belief_{k|k-1}^{\th} = \nabla_{\th(l)}\one^\p\belief_{k|k-1}^{\th}  = \nabla_{\th(l)} 1 = 0$.

We need a recursion for evaluating $w_k^{(l)}(\th)$ in (\ref{eq:score}).
Differentiating $\belief^{\th}_{k+1|k}$ with respect to $\th(l)$
 yields
\beq w_{k+1}^{(l)}(\th) = \nabla_{\th(l)}  \belief^{\th}_{k+1|k}= 
R_1(\obs_k,\belief^{\th}_{k|k-1},\th) w_k^{(l)}(\th)
+ R_2^{(l)}(\obs_k,\belief^{\th}_{k|k-1},\th)  \label{eq:wgrad}\end{equation}
where
\begin{align*} R_1(\obs_k,\belief^{\th}_{k|k-1},\th) & =
\tp'(\th) \left[I - \frac{\oprob_{\obs_k}(\th)\,\belief^{\th}_{k|k-1}{\bf 1}^{\prime}}{
{\one}^{\prime}\oprob_{y_k}(\th)\,\belief^\th_{k|k-1}}\right] \frac{\oprob_{\obs_k}(\th)}{
{\one}^{\prime}\oprob_{y_k}(\th)\,\belief^\th_{k|k-1}} \\
 R_2^{(l)}(z_k,\belief^{\th}_{k|k-1}, \th) & =
\tp^{\prime}(\th)\left[I- \frac{\oprob_{\obs_k}(\th)\,\belief^{\th}_{k|k-1}{\bf 1}^{\prime}}{
{\one}^{\prime}\oprob_{y_k}(\th)\,\belief^\th_{k|k-1}}\right]
\frac{\one^\p \nabla_{\th(l)} \oprob_{\obs_k}(\th)\, \belief^{\th}_{k|k-1}}{
{\one}^{\prime}\oprob_{y_k}(\th)\,\belief^\th_{k|k-1}} \\
&\quad\disp + \frac{[ \nabla_{\th(l)} \tp^{\prime}(\th)]
\, \oprob_{\obs_k}(\th)\,\belief^{\th}_{k|k-1}}{
{\one}^{\prime}\oprob_{y_k}(\th)\,\belief^\th_{k|k-1}}. \end{align*}

To summarize (\ref{eq:rmlesa}), (\ref{eq:score}) and (\ref{eq:wgrad}) constitute the RMLE algorithm for online parameter estimation of an HMM.

The ODE analysis of the above algorithm is in \cite{KY02}. Recursive prediction and EM type algorithms were developed in the 1990s  \cite{KM93,CKM94}; see also
\cite{CMR05} for a detailed description.


\section{Example 2: HMM State Estimation via LMS  Algorithm} \label{sec:hmmsa}
 \index{stochastic approximation algorithm! HMM state estimation}
 \index{HMM! state estimation by LMS algorithm}
 \index{LMS algorithm! HMM state estimation|(}
 
 This section discusses how the least mean squares (LMS)  stochastic gradient  algorithm can  be used to estimate the underlying state of a slow Markov chain
 given noisy observations.  Unlike  the HMM
 filter, the LMS algorithm does not require exact  knowledge of the underlying transition matrix or observation probabilities.
  In the context of POMDPs, this implies that for slow Markov chains, the underlying state 
  can be estimated (with provable performance bounds) and then an MDP controller can be run. This approach
lies within the class of 
 MDP based heuristics for solving POMDPs.

\subsection{Formulation} \label{sec:form}
Let $\{y_n\}$ be a sequence of real-valued signals representing the
observations obtained at time $n$, and  $\{\thstate_n\}$ be the
time-varying true parameter, an $\reals^r$-valued random process.
Suppose that
\beq \label {ob} y_n =\ph'_n \thstate_n +\onoise_n, \ n=0,1, \ldots, \eeq
 where $\ph_n \in
\reals^r$ is the regression vector and $\{\onoise_n\} \in \reals$ is a
 zero mean sequence.
Note that (\ref{ob}) is a variant of the usual linear regression
model, in which, a time-varying stochastic
process $\thstate_n$ is in place of a fixed parameter.
We assume that $\thstate_n$ is a slow
discrete-time Markov chain.


\begin{itemize}
\item[(A1)] Suppose that there is a small parameter $\e>0$ and that $\{\thstate_n\}$
is a  Markov chain with states
and transition probability matrix  given by
\beq \label{statespace} \statespace= \{ 1, 2, \ldots, \statedim\},  \quad  \text{ and } \quad 
\tp^\e= I+ \e Q.\eeq
Here $I$ denotes the  $\statedim \times \statedim$
identity matrix and $Q=(Q_{ij})$ is a $\statedim \times \statedim$ 
generator matrix of a continuous-time Markov
chain (i.e., $Q$ satisfies $Q_{ij} \ge 0$ for $i\not =j$ and
$\sum^\statedim_{j=1}Q_{ij}=0$ for each $i=1,\ldots,\statedim$).
For simplicity, assume the initial
distribution $\P(\thstate_0= g_i)= \belief_{0}(i)$
to be independent of $\e$
for each $i=1,\ldots,\thstatedim$.
\end{itemize}

 Note that the small parameter
$\e>0$ in (A1) ensures that the identity matrix $I$ dominates. In
fact, $Q_{ij}\ge 0$ for $i\not =j$ thus the small parameter $\e>0$
ensures the entries of the transition matrix to be positive since
$\tp^\e_{ij}= \delta_{ij} +\e Q_{ij}\ge 0$ for $\e>0$ small
    enough, where $\delta_{ij}=1$ if $i=j$ and is 0 otherwise.
The use of the generator $Q$ makes the row sum of the
    matrix $P$ be one since $\sum_{j=1}^\thstatedim \tp^\e_{ij}= 1 + \e
    \sum^\statedim_{j=1}Q_{ij} =1.$ The essence is that although the true
    parameter is time varying,
    it  is piecewise constant.  In addition,
    the process does not change too frequently due to the dominating
    identity matrix in the transition matrix (\ref{ob}).  It remains
    as a constant most of the time and jumps into another state
    at random instants. Hence the terminology ``slow'' Markov chain.

\paragraph{LMS Algorithm.} 
 \index{LMS algorithm} \index{adaptive filtering}
The LMS algorithm will be used to track the Markov chain  $\{\thstate_n\}$.  Recall from (\ref{eq:lms1}) that the LMS algorithm generates estimates 
$\{\wdh \th_n\}$ according to
\beq \label{alg} \wdh \th_{n+1} =\wdh \th_n + \mu \ph_n(y_n-\ph'_n \wdh \th_n), \quad
\ n=0,1,\ldots, \eeq
where $\mu>0$ is  a small constant step size for the algorithm.
By using (\ref{ob}) with $\wdt \th_n =\wdh \th_n -\thstate_n$, the tracking error satisfies
\beq \label{alg-eq} \wdt \th_{n+1}= \wdt \th_n - \mu \ph_n \ph'_n \wdt \th_n
+ \mu \ph_n \onoise_n + (\thstate_n -\thstate_{n+1}).\eeq 
The aim is to determine
 bounds on the deviation $\wdt \th_n= \wdh \th_n -
\thstate_n$. This goal is accomplished by the following four steps:

\begin{enumerate}
\item Obtain  mean square error bounds for $\E | \wdh \th_n -\thstate_n|^2$.
\item Obtain a limit ODE of centered process.
\item Obtain a weak convergence result of a suitably scaled sequence.

\end{enumerate}

The Markov chain  $\thstate_n$ is called a  {\em hypermodel} in \cite{BMP90}. While
 the dynamics of the hypermodel $\thstate_n$  are used in our
analysis, it does
not enter the  implementation of the
LMS algorithm (\ref{alg}) explicitly.

\paragraph{Assumptions on the Signals.}
Let ${\cal F}_n$ be the $\sigma$-algebra generated
by $\{ (\ph_j,\onoise_j),\ j<n, \thstate_j, \ j \le n\}$, and denote the conditional
expectation with respect to ${\cal F}_n$ by $\E_n$.
We will use the following
conditions on the signals.

\begin{itemize}
\item[(A2)] The signal $\{\ph_n,\onoise_n\}$ is independent of $\{\thstate_n\}$.
Either $\{\ph_n,\onoise_n\}$ is a
 sequence of bounded signals such that
there is a symmetric and positive definite
matrix $B \in \reals^{r\times r}$
 such that $\E\{ \ph_n \ph'_n\} =B$
 \beq \label{ineq-mix-1} {\Big|} \sum^{\infty} _{j=n} \E_n \{\ph_j \ph'_j - B \}
 {\Big|} \le K, \; \; \text{ and } 
{\Big|} \sum^{\infty} _{j=n} \E_n \{\ph_n \onoise_j\}
 {\Big|}\le K , \eeq
 or $\{\ph_n,\onoise_n\}$
 is a  martingale difference 
satisfying $\sup_nE\{|\ph_n|^{4+\Delta}\}<\infty$ and
$\sup_nE\{|\ph_n \onoise_n|^{2+\Delta}\} <\infty$
for some $\Delta>0$.

\end{itemize}

{\em Remark}:
Inequalities (\ref{ineq-mix-1}) are modeled after mixing processes and are in the
almost sure (a.s.) sense with the constant $K$ independent of
$\omega$, the
sample point. [Note, however,
we use the same kind of notation as
in, for example, the mixing inequalities \cite[p.166, Eq.
(20.4)]{Bil68} and
\cite[p.82, Eqs. (6.6) and (6.7)]{Kus84}.]
This allows us
to work with correlated
signals whose  remote past and distant future are asymptotically
independent. To obtain the desired result, the distribution of the
    signal need not be known. The boundedness is a mild restriction,
    for example, one may consider truncated Gaussian processes etc.
    Moreover, dealing with recursive procedures
    in practice, in lieu of
    (\ref{alg}), one often uses a projection or truncation algorithm.
    For instance, one may use \beq \label{proj}\wdh \th_{n+1}= \pi_H[ \wdh
    \th_n + \mu \ph_n (y_n-\ph'_n \wdh \th_n)],\eeq where $\pi_H$ is a
    projection operator and $H$ is a bounded set. When the iterates
    are
    outside $H$, it will be projected back to the constrained set
    $H$. Extensive discussions for such projection algorithms can be
    found in \cite{KY03}.

\subsection{Analysis of Tracking Error} \label{sec:trackingerror}

\subsubsection{Mean Square Error Bounds}
\label{sec:meansq}
This section establishes a mean square error estimate for $\E | \wdh
\th_n -\thstate_n|^2$.  
It is important to note that the mean square error
analysis below holds for
small positive but fixed $\mu$ and $\e$. Indeed, let
$\lambda_\text{min}>0$ denote the smallest eigenvalue
of the symmetric positive definite  matrix $B$ defined in (A2).
Then in the following theorem,
it is sufficient to pick out $\mu$ and $\e$  small enough
so that
\beq \label{eq:mesmall}
 \lambda_\text{min} \mu > O(\mu^2) + O(\e \mu);  \eeq
The phrase ``for sufficiently large $n$'' in what follows means that
there is an $n_0=n_0(\e,\mu)$ such that (\ref{eq-msq}) holds for $n
\ge n_0$. In fact,  (\ref{eq-msq}) holds uniformly for $n\ge n_0$.

\begin{theorem}  \label{MSE}
  Under conditions {\rm (A1)} and {\rm (A2)}, for
  sufficiently large $n$, as $\e\to 0$ and $\mu\to 0$,
  \beq \label{eq-msq} \E | \wdt \th_n |^2 = \E | \wdh
  \th_n -\thstate_n|^2 = O\(\mu + \e/\mu\) \exp( \mu + \e /\mu).\eeq
\end{theorem}

The theorem is proved in the appendix to this chapter.

In view of Theorem \ref{MSE}, it is clear that in order
for the adaptive algorithm to track the time-varying parameter, due to the presence of the term $\exp(\e/\mu)$, we need to have at least $\e /\mu =O(1)$.
Thus, the ratio $ \e/\mu$ must not be large.
A glance of the order of magnitude estimate
$O(\mu + \e /\mu)$, to balance the two terms $\mu$ and $\e /\mu$, we need to
choose $\e = O(\mu^2)$. Therefore, we arrive at the following corollary.

\begin{corollary}  \label{MSE-cor}
Under the conditions of Theorem \ref{MSE}, if $\e =O(\mu^2)$,
then for sufficiently large $n$,
$ \E | \th_n|^2 = O(\mu).$
\end{corollary}

\subsubsection{Mean ODE and Diffusion Approximation} \label{sec:ode}
  Due to the form of the transition
matrix given by (\ref{statespace}), the underlying Markov chain
belongs to the
category of two-time-scale
Markov chains. For some of the recent
work on this subject, we refer the reader to \cite{YZ06} and
the references therein.
It is assumed that  $\e =O(\mu^2)$ (see
Corollary \ref{MSE-cor}), i.e., the adaptation speed of the
LMS algorithm (\ref{alg}) is faster than the Markov chain
dynamics. 

Recall that the mean square error analysis in
\secn \ref{sec:meansq} deals with the mean square behavior of the
random variable $ \wdt \th_n =\wdh \th_n - \thstate_n$ as $n\rightarrow
\infty$, for small but fixed $\mu$ and $\e$.  In contrast, the
mean ODE and diffusion approximation analysis of this section deal
with how the entire discrete-time trajectory (stochastic process)
$\{\wdt \th_n: n=0,1,2\ldots,\}$ converges (weakly) to a the
limiting continuous-time process (on a suitable function space) as
$\mu \rightarrow 0$ on a time scale $O(1/\mu)$.  
Since the
underlying true parameter $\thstate_n$ evolves according to a Markov
chain (unlike standard stochastic approximation proofs where the
parameter is assumed constant), the proofs of the ODE and
diffusion limit  are non-standard and
require use of the so called ``martingale problem'' formulation
for stochastic diffusions; please see \cite{YK05b,YKI04} for details.

\subsubsection{Mean Ordinary Differential Equation (ODE)} \label{sec:sode}
Define 
$\wdt \th_n =\wdh \th_n
- \thstate_n$ and the interpolated process
\beq \label{vmu}\wdt \th^\mu(t)= \wdt \th_n
\quad \text{ for } t\in [n\mu, n\mu + \mu). \eeq
We will need another condition.

\begin{itemize}
\item[(A3)]
 As $n\to \infty$,
\begin{align*} \ad {1\over n}
\sum^{n_1+n}_{j=n_1} \E_{n_1} \ph_j \onoise_j
\to 0, \ \hbox{ in probability}, \quad 
\ad {1\over n} \sum^{n_1+n}_{j=n_1} \E_{n_1} \ph_j \ph'_j
\to B, \ \hbox{ in probability}. \end{align*}
\end{itemize}

\begin{theorem}    \label{ode-lim}
Under {\rm (A1)}--{\rm (A3)} and assuming
that $\wdt \th_0 = \wdt \th^\mu_0$
converges weakly to $\wdt\th^0$, then
$\wdt \th^\mu\cd$ defined in {\rm(\ref{vmu})}
converges weakly to $\wdt \th \cd$, which is a solution of the
ODE
\beq \label{ode-lim-eq} \frac{d}{dt} \wdt \th (t)=- B \wdt \th(t), \ t \ge
0, \  \wdt\th(0) = \wdt \th^0 .\eeq
\end{theorem}

 This theorem provides us with the evolution of the tracking
errors. It shows that
$ \wdh\th_n - \thstate_n$ evolves dynamically so that its trajectories
follows a deterministic ODE. Since the ODE is asymptotically stable,
the errors decay exponentially fast to 0 as time grows.

\subsection{Closing Remark. Tracking Fast Markov Chains} \label{sec:fastmc} 

 What happens if $\epsilon = O(\mu)$ so that the Markov chain evolves on the same time scale as
the LMS  algorithm? The analysis is more complex \cite{YKI04} and the main results are as follows:
\begin{compactitem}
\item The mean square tracking error is (compare with Theorem \ref{MSE-cor})
 \beq  \E | \wdt \th_n |^2 = \E | \wdh
  \th_n -\thstate_n|^2 = O(\mu ).\eeq
\item The ODE  becomes a Markov modulated ODE (compare with Theorem \ref{ode-lim})
$$
\frac{d}{dt} \wdt \th (t)=- B \wdt \th(t) + Q \state(t) , \quad \ t \ge
0, \  \wdt\th(0) = \wdt \th^0 .
$$
where $x(t)$ is the interpolated Markov process (constructed as in (\ref{vmu})) with transition rate matrix $Q$.
So stochastic averaging no longer results in a deterministic ODE. Instead the  limiting behavior is  specified by a differential equation driven by a continuous-time Markov chain with
transition rate matrix $Q$, see \cite{YKI04,YIK09}.
\end{compactitem}
 \index{LMS algorithm! HMM state estimation|)}

\renewcommand{\f}{\psi}
\def\fhat{\hat{\f}}

\section{Example 3: Discrete Stochastic Optimization for  Policy Search}  \label{sec:dopt} \index{discrete stochastic optimization|(}

In this section we consider 
a POMDP modeled as a back box: a decision maker can choose one of  a finite number of policies and can measure the noisy costs incurred over
a horizon. 
The aim is to
solve the following time-varying discrete stochastic optimization problem:
\beq \text{ Compute } \th^* =  \argmin_{\theta \in \Theta}  C(\th) , \quad \text{ where }
C(\th) = \E \{ \cost_n(\th) \}  .
  \label{eq:discobj}
\eeq
Here
\begin{compactitem}
\item $\Theta = \{1,2,\ldots, S\}$  denotes the finite discrete space of possible policies.
 \item $\theta$ is a  parameter that specifies a  policy - for example it could specify how long and in which order to use specific sensors
 (sensing modes) in a controlled sensing problem.
\item $c_n(\theta)$ is the  observed cost incurred when using
strategy $\theta$ for sensing. 
Typically, this cost is evaluated by running the POMDP over a specified horizon. So $n$ can be viewed as an index for the $n$-th run of the POMDP. 

\item $\th^* \in \globalopt$ where $\globalopt \subset \Theta$ denotes the  set of global minimizers of (\ref{eq:discobj}).
\end{compactitem}

It is assumed that  the probability distribution of the noisy cost $\cost_n$  is not known. 
Therefore, 
$C(\th)$ in (\ref{eq:discobj}) cannot be evaluated at each time $n$; otherwise  the problem is a deterministic integer optimization problem.  
We can only observe noisy samples  $\cost_n(\th)$ of the system performance $C(\th)$ via  simulation for any choice of $\th \in \Theta$.
Given this noisy performance   $\{\cost_n(\th)\}$, $n=1,2,\ldots,$ the aim is to estimate $\th^*$.
An obvious  brute force method  for solving  (\ref{eq:discobj}) involves an exhaustive enumeration: For each $\th \in \Theta$, compute  the
empirical average
$$ \wdh{c}_N(\th) = \frac{1}{N} \sum_{n=1}^N c_n(\th),$$
via simulation
for large $N$.
Then pick
$\wdh{\th} = \argmin_{\th \in \Theta}
\wdh{c}_N(\th)$.
Since for any fixed $\th\in \Theta$,
$\{c_n(\th)\}$ is an  i.i.d.\
sequence of  random variables, by virtue of
Kolmogorov's  strong law of large numbers,
$\wdh{c}_N(\th) \rightarrow \E \{c_n(\th)\}$ w.p.1,
 as $N\rightarrow \infty$.
This and the finiteness of $\Theta$ imply that 
\beq
\argmin_{\th \in \Theta}
\wdh{c}_N(\th) \rightarrow
\argmin_{\th \in \Theta}
\E\left\{c_n(\th) \right\}  \text{ w.p.1} \text{ as }  N\rightarrow \infty . \label{eq:empirical}\eeq
In principle,
the above  brute force simulation method can  solve the
discrete stochastic optimization
problem (\ref{eq:discobj}) and the estimate is {\em consistent}, i.e.,
(\ref{eq:empirical}) holds. However, the method is highly inefficient
since  $\wdh{c}_N(\th) $ needs to be evaluated for each $\th\in\Theta$.
The evaluations of $\wdh{c}_N(\th) $ for $\th \not\in \globalopt$
are  wasted because they contribute nothing to  the estimation
of $\wdh{c}_N(\th) $, $\th \in \globalopt$.

The idea of discrete stochastic approximation  is to design an
algorithm that is both {\em consistent} and {\em attracted} to the
minimum.
That is, the algorithm should spend more time
obtaining observations $c_n(\th)$ in areas of the state
space near $\globalopt$ 
 and less in
other areas.
%

 \index{discrete stochastic optimization! multi-armed bandit}
The discrete stochastic optimization  problem (\ref{eq:discobj})
is similar in spirit to the  stochastic bandit problem  with the key
difference that (\ref{eq:discobj}) aims to determine the minimizer $\th^*$  while bandits aim to minimize an average observed function over a period of time (called the regret).

In the remainder of this section we provide two examples of  discrete stochastic optimization algorithms:
\begin{compactitem}
\item The smooth best-response adaptive search algorithm.
\item Discrete stochastic search algorithm.
\end{compactitem}
A numerical example is then given to illustrate the performance of these algorithms. Also the performance is compared with
 the 
{\em upper confidence bound (UCB) algorithm}  which is a popular multi-armed bandit algorithm.

We are interested in situations where the objective function $C(\th)$ in (\ref{eq:discobj})  evolves randomly over a slow time scale and hence
the global minimizers  evolve slowly with time. So the  stochastic optimization
algorithms described below have constant step sizes to facilitate tracking time varying global minimizers.
\index{random search algorithms}

\subsection{Algorithm 1. Smooth Best-Response Adaptive Search}

\index{discrete stochastic optimization! best-response search}
Algorithm \ref{alg:adaptive-search} describes the smooth best-response adaptive search algorithm for solving the discrete stochastic
optimization problem (\ref{eq:discobj}).  In this algorithm, $\gamma \in (0,\maxdiff]$ denotes the exploration weight where  $\maxdiff \geq  \max_\th C(\th) - \min_\th C(\th)$ upper bounds
the maximum difference in objective function  among the feasible solutions. \index{best response adaptive search}

\begin{algorithm}
\caption{Smooth Best-Response Adaptive Search}

\noindent \textbf{\textit{Step 0.} Initialization.} 
Set  exploration parameter $\gamma \in (0,\maxdiff]$. 
Initialize $\f_{_0} = \mathbf{0}_{S}.$

\textbf{\textit{Step 1.} Sampling.} Sample $\th_n\in \Theta$ according to $S$-dimensional probability vector
$$\br(\f_n)  
 =\lb b^\gamma_1(\f_n),\ldots,b^\gamma_S(\f_n)\rb, \quad \text{ where } \bbr_i\big(\f\big) =
 \frac{\exp\big(- \psi_i/\gamma\big)}{\sum_{j\in\Theta}\exp\big(- \psi_j/\gamma\big)}. $$


\textbf{\textit{Step 2.} Evaluation.} Perform simulation to obtain $\cost_n(\th_n)$.

\textbf{\textit{Step 3.} Belief Update.} Update $\f_{n+1} \in \reals^S$ as
\begin{align} \label{eq:regret-update}
\f_{n+1} &= \f_n + \mu\left[\boldf\big(\th_n,\f_n,\cost_n(\th_n)\big) - \f_n\right],
\\ \text{ where }\;
 \boldf &= (f_{_{1}},\ldots,f_{_{S}})^\p, \quad
f_{_{i}}\big(\th_n,\f_n,\cost_n(\th_n)\big) = \frac{\cost_n(\th_n)}{\bbr_i\big(\f_n\big)}  I(\th_n = i).  \nn
\end{align}

\textbf{\textit{Step 4.} Recursion.} Set $n\leftarrow n+1$ and go to Step 1.
\label{alg:adaptive-search}
\end{algorithm}

The $S$-dimensional probability vector $\br(\f_n)  $ in Step 1 of Algorithm  \ref{alg:adaptive-search} is actually 
a special case of the   following
general framework: Let $\simplex$ denote the unit $S$-dimensional simplex, and $\text{int}(\simplex)$ denote the interior of this simplex.
Given the
vector $\f = [\psi_{1},\ldots,\psi_S]^\p\in\reals^S$ of beliefs about objective values at different candidate solutions, the \emph{smooth best-response sampling} strategy $\br \in \simplex$  is 
\begin{equation}
\label{eq:smooth-br}
\br\big(\f\big) := \operatorname*{argmin}_{\ {\sigma} \in \simplex } \textstyle \left(  \sigma^\p \f - \gamma \rho( {\sigma})\right),\;\; 0<\gamma<\maxdiff,
\end{equation}
where $\sigma  \in \simplex$  and the perturbation function $\rho( {\sigma}):\textmd{int}(\simplex )\to \reals$ satisfy
\begin{itemize}
    \item[i)] $\rho\cd$ is  continuously differentiable, strictly concave, and $|\rho|\leq 1$;
    \item[ii)] $\|\nabla\rho( {\sigma})\|\to\infty$ as $ {\sigma}$ approaches the boundary of $\simplex$  
    where $\|\cdot\|$ denotes the Euclidean norm. 
\end{itemize}
In Algorithm \ref{alg:adaptive-search}, we chose  $\rho\left(\cdot\right)$ as the \emph{entropy function}~\cite{FL98}
$
\rho\left( {\sigma}\right) = -\sum_{i\in\Theta} \sigma_i \log \left(\sigma_i\right)
$. Then applying (\ref{eq:smooth-br}) yields Step 1 which can be viewed as 
a  Boltzmann exploration strategy with constant temperature
\begin{equation}
\label{eq:boltzmann}
\bbr_i\big(\f\big) = \frac{\exp\big(- \psi_i/\gamma\big)}{\sum_{j\in\Theta}\exp\big(- \psi_j/\gamma\big)}.
\end{equation}
Such an exploration strategy is  used widely in the context of game-theoretic learning and is called logistic fictitious-play~\cite{FL95} or logit choice function~\cite{HS02}.

 \cite{NKY14} shows that 
that the sequence $\{\th_n\}$ generated by Algorithm \ref{alg:adaptive-search} spends more time in the set of global
maximizers $\globalopt$ than any other point in $\Theta$.

\subsection{Algorithm 2. Discrete Stochastic  Search}
\index{discrete stochastic optimization! random search}
The second discrete stochastic optimization algorithm that we discuss is displayed in
Algorithm  \ref{alg:RS}. This random search algorithm was  proposed by
Andradottir \cite{And96,And99}  for computing
the global minimizer in (\ref{eq:discobj}).   Recall $\globalopt$ denotes the set of global minimizers of (\ref{eq:discobj}) and $\Theta = \{1,2,\ldots,S\}$ is the search space. The following assumption is required:

\begin{itemize}
\item[{(O)}]
 For each $\th , \tilde{\th} \in \Theta- \globalopt$ and $\th^* \in \globalopt$,
\begin{align*}
\prob(\cost_n(\th^*)  < \cost_n(\th) ) & \geq \prob(\cost_n(\th) > \cost_n(\th^*)) \\
   \prob(\cost_n(\tilde{\th})  > \cost_n(\th^*) ) & \geq \prob(\cost_n(\tilde{\th}) > \cost_n(\th))
\end{align*}
\end{itemize}

Assumption (O) ensures that the algorithm is more likely
to jump towards a global minimum than away from it. 
Suppose
 $c_n(\th) = \th + w_n(\th)$  in
(\ref{eq:discobj}) for each $\th \in \Theta$, where
$\{w_n(\th)\}$ has a symmetric probability density function or probability mass function  with zero mean.
Then 
 assumption (O) holds.

\begin{algorithm}
\textbf{\textit{Step 0.} Initialization.} Select $\th_0\in\Theta$. Set $\pi_{i,0} = 1$ if $i = \th_0$, and $0$ otherwise. \\ Set $\th^\ast_0 = \th_0$ and $n=1$.

\textbf{\textit{Step 1.} Sampling.} Sample a candidate solution $\tilde{\th}_n$ uniformly from the set $\Theta-\th_{n-1}$.

\textbf{\textit{Step 2.} Evaluation.} 
Simulate samples $\cost_n(\th_{n-1})$ and $\cost_n(\tilde{\th}_n)$.
  \\ If  $\cost_n(\tilde{\th}_n) < \cost_n(\th_{n-1})$, set $
 \th_{n} = \tilde{\th}_n$, else,
set $\th_{n} = \th_{n-1}$.


\textbf{\textit{Step 3.} Belief (Occupation Probability) Update.}
\begin{equation}
\label{eq:RS}
 {\pi}_n =  {\pi}_{n-1} + \mu\lb e_{\th_n} -  {\pi}_{n-1}\rb,\quad 0<\mu<1,
\end{equation}
where $e_{\th_n}$ is the unit vector with the $\th_n$-th element being equal to one.

\textbf{\textit{Step 4.} Global Optimum Estimate.} $\th^\ast_n \in \argmax_{i\in\Theta} \pi_{i,n}$.\\
\textbf{\textit{Step 5.} Recursion.} Set $n \leftarrow n + 1$ and go to Step 1.

\caption{Random Search (RS) Algorithm}
\label{alg:RS}
\end{algorithm}

 Algorithm \ref{alg:RS} is attracted to the global minimizer set $\globalopt$ in the sense that it spends more time
in $\globalopt$ than at any other candidate in $\Theta - \globalopt$. This was proved in in \cite{And96}; see also \cite{KWY04,YKI04} for tracking time varying optima.

\subsection{Algorithm 3. Upper Confidence  Bound  (UCB) Algorithm} \index{upper confidence bound (UCB) algorithm}
\index{discrete stochastic optimization! upper confidence bound algorithm}
\index{multi-armed bandit! UCB algorithm}
The final algorithm that we discuss is the UCB algorithm.
UCB  \cite{ABF02} belongs to  the family of ``follow the perturbed leader'' algorithms and is widely used in the context of multi-armed bandits.
Since bandits are typically formulated in terms of maximization (rather than minimization), we consider here {\em maximization}
of $C(\th)$ in (\ref{eq:discobj}) and denote  the global maximizers as $\th^*$.
 Let $B$ denote an upper bound on the objective function and $\xi > 0$ be a constant. The UCB algorithm is summarized  in Algorithm~\ref{alg:UCB}. For a static discrete stochastic optimization problem, we set $\mu = 1$;
otherwise, the discount factor $\mu$ has to be chosen in the interval $(0,1)$. Each iteration of  UCB  requires $O(S)$ arithmetic operations, one maximization and one simulation of the objective function.

\begin{algorithm}
\textbf{\textit{Step 0.} Initialization.} Simulate each  $\th \in\Theta = \{1,2,\ldots,S\}$ once to obtain $\cost_1(\th)$.\\  
Set $\widehat{\cost}_{\th,S} = \cost_1(\th)$ and 
$m_{\th,S} = 1$. 
Set $n=S+1$.

\textbf{\textit{Step 1a.} Sampling.} At time $n$  sample  candidate solution
\begin{align*}
\textstyle \th_n &= \argmax_{\th\in\Theta} \left[\widehat{\cost}_{\th,n-1} + B\sqrt{\frac{\xi\log (M_{n-1}+1)}{m_{\th,n-1}}}  \right] \\
 \text{ where } &
\widehat{\cost}_{\th,n-1} \textstyle= \frac{1}{m_{\th,n-1}} \sum_{\tau = 1}^{n-1} \mu^{n-\tau-1} \cost_{\tau}(\th_{\tau}) I(\th_\tau = \th), \\
m_{\th,n-1} &= \sum_{\tau=1}^{n-1} \mu^{n-\tau-1} I(\th_{\tau} = \th), \quad 
M_{n-1} = \sum_{i=1}^S m_{i,n-1} 
\end{align*}
 \textbf{\textit{Step 1b}.  Evaluation.}  Simulate to obtain $\cost_n(\th_n)$.

\textbf{\textit{Step 2.} Global Optimum Estimate.} $\th^\ast_n \in \argmax_{i\in\Theta} \widehat{\cost}_{i,n}$. \\
\textbf{\textit{Step 3.} Recursion.} Set $n \leftarrow n + 1$ and go to Step 1.

\caption{Upper confidence bound (UCB) algorithm for maximization of objective $C(\th)$, $\th \in \Th = \{1,2,\ldots,S\}$ with discount factor $\mu \in 
(0,1]$ and exploration constant $\xi >0$. $B$ is an upper bound on the objective function.}
\label{alg:UCB}
\end{algorithm}

\subsection{Numerical Examples} This section illustrates the performance of the above discrete stochastic optimization algorithms in estimating the mode of an unknown  probability mass function. 
 Let $\degdist(\th)$,  $\th \in \Th = \{0,1,\ldots,S\}$ denote the degree distribution\footnote{The degree of a node is the number of  connections or edges the node has to other nodes. 
The degree distribution $\degdist(\th)$ is the fraction of nodes in the network with degree $\th$, where $\th \in \Th = \{0,1,\ldots\}$. Note that $
\sum_{\th} \degdist(\th) = 1$ and so the
degree distribution is  a pmf with support on $\Th$.}   of a social network (random graph). Suppose that a pollster aims to estimate the mode of the degree
distribution, namely \beq  \th^* = \argmax_{\th \in \lbr 0,1,\ldots,S\rbr} \degdist(\th). \label{eq:modedef} \eeq
The pollster does not know  $\degdist(\th)$.
It  uses the following protocol to estimate the mode.
At each time $n$, the pollster chooses a specific  $\th_n \in \Th$, and then 
 asks a randomly sampled individual in the social network: Is your degree $\th_n$? The individual replies ``yes" (1)  or ``no" (0). Given these 
 responses $\{I(\th_n), n=1,2,\ldots\}$ where $I(\cdot)$ denotes the indicator function,
the pollster aims to solve the discrete stochastic optimization problem: Compute 
\begin{equation}
\label{eq:example-2}
\th^* = \argmin_{\degreelink \in \lbr 0,1,\ldots,S\rbr} - \E\lbr I( \degreelink_n)\rbr.
\end{equation}
Clearly the global optimizers of (\ref{eq:modedef}) and (\ref{eq:example-2}) coincide.
Below, we illustrate the performance of discrete stochastic optimization 
 Algorithm  \ref{alg:adaptive-search}  (abbreviated as AS), 
Algorithm \ref{alg:RS} (abbreviated as RS), and Algorithm \ref{alg:UCB} (abbreviated as UCB) for estimating $\th^*$ in (\ref{eq:example-2}).


In the numerical examples below,
 we simulated  the degree distribution as a Poisson pmf with parameter $\lambda$:
%
%
%
%
$$ 
\degdist(\model) = 
 \frac{{\lambda^\degreelink \exp(-\lambda)}}{\degreelink!},  \quad 
 \degreelink \in \{0,1,2,\ldots,S\}.
$$
It is straightforward to show that the 
mode of the  Poisson distribution is  
\beq \degreelink^* = \argmax_{\degreelink} \degdist(\degreelink)=
 \lfloor \lambda \rfloor. 
 \label{eq:groundtruth}
 \eeq If $\lambda$ is integer-valued, then  both $\lambda$ and $\lambda-1$ are modes of the Poisson pmf.
 Since the ground truth $\model^*$ is explicitly given by (\ref{eq:groundtruth}) we
have a simple benchmark
numerical study. (The algorithms do not use the knowledge that $\degdist$ is Poisson.)

\paragraph{Example: Static Discrete Stochastic Optimization.}
Consider the case where the Poisson rate $\lambda$ is a constant.
We consider  two  examples: i) $\lambda = 1$, which implies that the set of global optimizers is $\globalopt = \lbr 0,1\rbr$, and ii) $\lambda = 10$, in which case  $\globalopt = \lbr 9,10\rbr$. For each case, we further study the effect of the size of search space on the performance of algorithms by considering two instances: i) $S=10$, and ii) $S = 100$. Since the problem is static in the sense that $\globalopt$ is fixed for each case, one can use the results of~\cite{BF13} to show that if the exploration factor $\gamma$ in~\eqref{eq:smooth-br} decreases to zero sufficiently slowly, the sequence of samples $\lbr \th_n\rbr$ converges almost surely to the global minimum. We consider the following modifications to AS Algorithm~\ref{alg:adaptive-search}:
\begin{compactenum}
    \item[(i)] The constant step-size $\mu$ in~\eqref{eq:regret-update} is replaced by decreasing step-size $\mu(n) = \frac{1}{n}$;
    \item[(ii)] The exploration factor $\gamma$ in~\eqref{eq:smooth-br} is replaced by $\frac{1}{n^\beta}$, $0<\beta<1$.
\end{compactenum}
We chose  $\beta = 0.2$ and $\gamma = 0.01$. Also in Algorithm~\ref{alg:UCB}, we set  $B = 1$, and $\xi = 2$.

\begin{table}[!t]
  \centering
\caption{Example 1: Percentage of Independent Runs of Algorithms that Converged to the Global Optimum Set in $n$ Iterations.}
  \subtable[$\lambda = 1$]{
  \label{table:1}
  \centering
\renewcommand{\arraystretch}{1.1}
\begin{tabular}{cccccccc}
\firsthline
Iteration & \multicolumn{3}{c}{$S = 10$} & & \multicolumn{3}{c}{$S = 100$} \\
\cline{2-4} \cline{6-8}
$n$       &  AS  &      RS     &     UCB     & &  AS   &    RS       &     UCB \\
\hline
10        &  55  &      39     &     86      & &  11   &     6       &      43     \\
50        &  98  &      72     &     90      & &  30   &     18      &      79     \\
100       &  100 &      82     &     95      & &  48   &     29      &      83     \\
500       &  100 &      96     &     100     & &  79   &     66      &      89     \\
1000      &  100 &     100     &     100     & &  93   &     80      &      91     \\
5000      &  100 &     100     &     100     & &  100  &     96      &      99     \\
10000     &  100 &     100     &     100     & &  100  &     100     &      100    \\
\lasthline
\end{tabular}
  }
  \subtable[$\lambda = 10$]{
  \label{table:2}
    \centering
\renewcommand{\arraystretch}{1.1}
\begin{tabular}{cccccccc}
\firsthline
Iteration & \multicolumn{3}{c}{$S = 10$} & & \multicolumn{3}{c}{$S = 100$} \\
\cline{2-4} \cline{6-8}
$n$       & AS & RS & UCB & &  AS   & RS  & UCB \\
\hline
10        &  29  &     14     &    15     & &   7    &     3       &     2      \\
100       &  45  &     30     &    41     & &   16   &     9       &     13     \\
500       &  54  &     43     &    58     & &   28   &     21      &     25     \\
1000      &  69  &     59     &    74     & &   34   &     26      &     30     \\
5000      &  86  &     75     &    86     & &   60   &     44      &     44     \\
10000     &  94  &     84     &    94     & &   68   &     49      &     59     \\
20000     &  100 &     88     &    100    & &   81   &     61      &     74     \\
50000     &  100 &     95     &    100    & &   90   &     65      &     81     \\
\lasthline
\end{tabular}
  }
\label{table-both}
\end{table}

Table~\ref{table-both} displays the performance of the algorithms AS, RS and UCB.
To  give a fair comparison,  the performance of the algorithms are compared based on
number of simulation experiments performed by each algorithm.
Observe the following from  Table~\ref{table-both}: In all three algorithms, the speed of convergence decreases when either $S$ or $\lambda$ (or both) increases. However, the effect of increasing $\lambda$ is more substantial since the objective function values of the worst and best states are closer when $\lambda = 10$.
Given equal number of simulation experiments, higher percentage of cases that a particular method has converged to the global optima indicates faster convergence rate. 

To evaluate and compare efficiency of the algorithms, the sample path of the number of simulation experiments performed on non-optimal feasible solutions is displayed in Figure \ref{fig:1}, when $\lambda = 1$ and $S=100$. As can be seen, since the RS method randomizes among all (except the previously sampled) feasible solutions at each iteration, it performs approximately 98\% of the simulations on non-optimal elements. The UCB algorithm switches to its exploitation phase after a longer period of exploration as compared to the AS algorithm.  
\begin{figure}[h]
\setlength{\abovecaptionskip}{0em}
\setlength{\belowcaptionskip}{0em}
     \centering
     \includegraphics[width=2.6in]{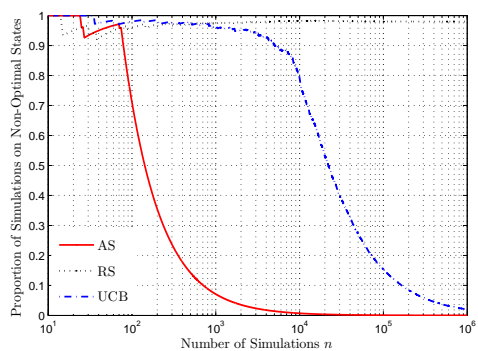}
     \caption{Example 1: Proportion of simulation effort expended on states outside the global optima set ($\lambda = 1$, $S = 100$).}
     \label{fig:1}
\end{figure}

In \cite{Kri16} we also present examples of how the algorithms perform when the optimum evolves according to a Markov chain.

\index{discrete stochastic optimization|)}

\section{Example 4: Mean Field Population Dynamics Models for  Social Sensing}  \index{social network! sensing}  \label{sec:popmf}
\index{social sensing}
\index{information diffusion in social network}
\index{diffusion! in social network|(}

We are interested in constructing tractable models for the diffusion of information over a social network comprising of a  population of interacting agents. 
As described in~\cite{Pin08}, such models arise in  a wide range of social phenomena such as diffusion of technological innovations,  ideas, behaviors, trends \cite{Gra78}, cultural fads, and economic conventions~\cite{Cha04} where
individual decisions are influenced by the decisions of others.  Once a tractable model has been constructed, 
Bayesian filters can be used to estimate the underlying state of nature \cite{KNH14}.

Consider a social network comprised of individuals (agents) who may or may not adopt a specific new technology (such as a particular brand of smartphone). We are interested in modeling how  adoption of  this technology
diffuses in the social network.
Figure \ref{fig:socialsensing} shows the schematic setup. 
The underlying state of nature $\{\nature_k\}$ can be viewed as the market conditions or competing  technologies that evolve with time and affect the adoption of the
technology.
The information (adoption of  technology) diffuses over the network --  the states of individual nodes (adopt or don't adopt) evolve  over time as a probabilistic function of the states of their neighbors and the   state of nature $\{\nature_k\}$.  Let $\pop_k$ denote the population state vector at time $k$; as explained below the $l$-th element of 
this vector, denoted $\pop_k(l)$, is  the fraction of the population with $k$ neighbors that has adopted the technology.
As the adoption of the new technology diffuses through the network, its effect is observed by social sensing -   such as user sentiment
on a micro-blogging platform like Twitter. 
  \index{sentiment analysis}
 The nodes that tweet their sentiments  act as social sensors.
 Suppose the state of nature $\nature_k$ changes   suddenly  due to a sudden market shock or presence of a new competitor.
 The goal for a market analyst or product manufacturer is to estimate the state of nature
  $\nature_k$  so as to detect the market shock or new competitor.

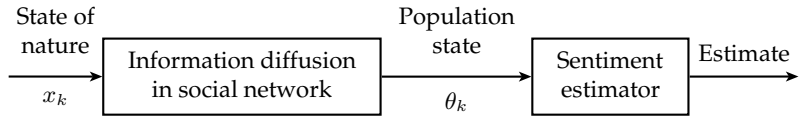
\begin{figure}[h]
\begin{pspicture}[showgrid=false](0.5,-0.25)(14,1.75)


\dotnode[dotstyle=square*,dotscale=0.001](12.5,0.75){dot}

\dotnode[dotstyle=square*,dotscale=0.001](2,0.75){H1}

\psblock(5.1,.75){B2}{\begin{tabular}{c} Information diffusion \\ in social network \end{tabular}}

\psblock(10,.75){B1}{\begin{tabular}{c} Sentiment  \\ estimator \end{tabular}}




\psset{style=Arrow}

\ncline{H1}{B2}  \naput[npos=.5]{\begin{tabular}{c} State of \\ nature\end{tabular}}
\nbput[npos=.5]{$\nature_k$}

\ncline{B2}{B1} \nbput[npos=.5]{$\pop_k$}
\naput[npos=.5]{\begin{tabular}{c}Population \\ state\end{tabular}}

\ncline{B1}{dot} \naput[npos=.5]{Estimate}

\end{pspicture}
\caption{Social Sensing of Sentiment} \label{fig:socialsensing}
\end{figure}

As a signal processing problem, the state of nature $\nature_k$   can be viewed as the signal (Markov chain), and the social network can be viewed as a sensor. The observed sentiment can be
viewed as an HMM: noisy measurement of $\pop_k$ (population state) which in turn depends on $\nature_k$ (state of nature).
The key difference compared to classical signal processing is that the social network (sensor) has dynamics due to the information diffusion over a graph. Estimating
$\nature_k$ can be formulated as a filtering problem, while
detecting sudden changes in $\nature_k$ can be formulated as a quickest detection problem. 

Dealing with large populations results in an combinatorial explosion in the size of the state space for $\pop_k$.  The aim of this section is to construct tractable models of the population dynamics by using the mean field
dynamics approximation.

\subsection{Population Model Dynamics} \label{sec:pop}
Let $\nature_k$ denote the state of nature at time $k$.
Consider  a population consisting of $\numagents $  agents indexed by $m=1,2,\ldots.\numagents$.
Each agent $m$ at time $k$ has  state 
$s_k^{(m)}\in \{1,2\ldots, L\}$.  Let the $L$ dimensional vector $\pop_k$ denote the fraction of agents in the $L$ different  states at time $k$.
We call $\pop_k$ the {\em population state}:
\beq  \pop_k(l)  =\frac{1}{\numagents}  \sum_{m=1}^\numagents I(s_k^{(m)}  = l) , \quad l =1,2,\ldots, L. \label{eq:popstate} \eeq
where $I(\cdot)$ denotes the indicator function.

Given the states $s_k^{(m)}$, $m=1,\ldots,M$ and hence  population state $\pop_k$ (\ref{eq:popstate}),
the population evolves as  a Markov process  in  the following 2-steps:
At   each time~$k$ \begin{compactitem} 
\item[{\em Step 1}.] An agent  $m_k$ is chosen uniformly from the $\numagents$ agents. The state $s_{k+1}^{(m_k)}$  at time $k+1$ of this agent
is simulated with
  transition probability
$$ \tp_{ij}(\nature_k,\pop_k) = \prob(s_{k+1} = j | s_k = i,\nature_k, \pop_k)$$ which depends on both the  state of nature and the population state.
\item[{\em Step 2}.] Then the population state $\pop_{k+1}$  reflects this updated state of agent $m_k$:
\beq  \pop_{k+1}  = \pop_k +  \frac{1}{\numagents} \left[ e_{s_{k+1}^{(m_k)}} - e_{s_k^{(m_k)}} \right]
\label{eq:popupdate}
\eeq  \end{compactitem}
Thus the transition probabilities for the  population process $\{\pop_k\}$ are
$$ \prob\bigl(\pop_{k+1}= \pop_k + \frac{1}{\numagents}[ e_j - e_i] \mid  \pop_k,\nature_k\bigr) =  \frac{1}{\numagents} \tp_{ij}(\nature_k,\pop_k) $$
Note that $\{\pop_k\}$ is an $\binom{\numagents + L -1}{L-1}$   state Markov chain with  state space 
$$ \popspace = \{ \pop: \sum_{l=1}^L \pop(l) =1, \;\pop(l) = n/L \text{ for some integer $n\geq 0$} \} .$$
Note also that $\popspace $ is a subset of the $L-1$ dimensional simplex denoted as $\I(L)$.

We are interested in modeling the evolution of the population process $\{\pop_k\}$ in (\ref{eq:popupdate}) when the number of agents
$\numagents$ is  large.  
\begin{theorem} \label{thm:mincpop}  For a population size of $\numagents$ agents,  where each agent has $L$ possible states, the population distribution process $\pop_k$ evolves in $\popspace \subset \I(L)$ as
\begin{align}
\pop_{k+1} &= \pop_k  + \frac{1}{\numagents}  H(\nature_k,\pop_k) +  \mtg_k , \quad \pop_0 \in \popspace,\label{eq:mincpop}
\\
 \text{ where  } &
 H(\nature_k,\pop_k) = \sum_{i=1}^L \sum_{j=1}^L  (e_j - e_i)  \tp_{ij}(\nature_k,\pop_k).  
 \nn \end{align}
Here $\mtg_k$ is an $L$ dimensional finite state martingale increment process with 
$ \| \mtg_k\|_2 \leq \Gamma/\numagents$ for some positive constant $\Gamma$.
\end{theorem}

\subsection{Mean Field Dynamics} \index{mean field dynamics} \label{sec:mfdpop}
The main result
below shows that for large $\numagents $,  the population process $\pop_k$ in (\ref{eq:mincpop}) converges to  a deterministic difference equation (or equivalently,
an ODE)  called the mean field dynamics. Thus the ODE method  serves the purpose of constructing  a tractable model for the population dynamics.

\begin{theorem}[Mean Field Dynamics] \label{thm:mfd}
Consider the deterministic mean field dynamics process with state $\bpop_k \in \I(L)$ (the $L-1$ dimensional unit simplex):
\beq \bpop_{k+1}  = \bpop_k  + \frac{1}{\numagents}  H( \nature_k,\bpop_k), \quad \bpop_0 = \pop_0.  \label{eq:mfd}  \eeq
Assume that $H(\th,\state)$ is  Lipschitz continuous\footnote{For any $\alpha, \beta \in \I(L)$, 
$\|H(\nature,\alpha) - H(\nature,\beta)\|_\infty \leq \lambda \| \alpha - \beta \|_\infty $ for some positive constant $\lambda$. \label{lip}}
 in $\th$.
Then for a time horizon of $\finaltime$ points,   the deviation between the mean field dynamics $\bpop_k$ in (\ref{eq:mfd}) and actual population distribution   $\pop_k$
in  (\ref{eq:mincpop}) satisfies
\beq \label{eq:mfdev}
\prob \bigl\{ \max_{0 \leq k \leq \finaltime} \left\| \bpop_k - \pop_k \right\|_\infty \geq \epsilon\bigr\}
\leq  C_1  \exp(-C_2 \epsilon^2 \numagents)
\end{equation}
providing $\finaltime = O(\numagents)$.
\end{theorem}

As an example, \cite{KNH14} uses the above formulation for the dynamics of information diffusion in a social network.

\section{\pwe} \label{sec:snbandit}
The ODE analysis for stochastic approximation algorithms was pioneered by Ljung  (see  \cite{Lju77,LS83}) and subsequently by Kushner and
co-workers \cite{Kus84,KY03}.  In this chapter we have only scratched the surface of this  remarkably powerful analysis tool. Apart from
the books listed above, \cite{BMP90,SK95b,Bor08} are also excellent references for analysis of such algorithms. The papers \cite{BHS05,BHS06} illustrate
the power of ODE method (and generalizations to differential inclusions) for analyzing the dynamics of game-theoretic learning.

\secn \ref{sec:hmmsa} uses the LMS algorithm for tracking parameters that jump infrequent but by possibly  large amounts.
In most traditional analyses,
the parameter changes by small amounts over small intervals of time. As mentioned in \secn \ref{sec:fastmc}, one can also analyze
the tracking capability for Markov chains that jump by large amounts on short intervals of time \cite{YKI04,YIK09}. In such cases, 
stochastic averaging leads to a
 Markov modulated ODE (instead of a deterministic ODE).  

It would be remiss of us not to mention the substantial literature in the analysis of adaptive filtering algorithms \cite{Say08,Hay13}.
The proof of Theorem \ref{MSE} uses perturbed Lyapunov functions. Solo \cite{SK95b} was influential in developing discounted perturbed Lyapunov function methods \index{Lyapunov function! perturbed}
for analysis of adaptive filtering algorithms.
The mean field dynamics proof of Theorem \ref{thm:mfd} is based on \cite{BW03} and uses Azuma-Hoeffding's inequality. It  requires far less mathematical machinery. 
The monograph \cite{Kur81} contains several deep results in Markov chain approximations of population processes and dynamics.

Survey papers  in discrete stochastic optimization include \cite{BMG09}. There has also been much recent work in using the ODE
analysis for ant-colony optimization algorithms and the use of ant-colony optimization in Q-learning \cite{DG14}.

We have omitted several important aspects of  stochastic approximation algorithms. One particular
intriguing result is  Polyak's iterate averaging \cite{PJ92}. By choosing a larger
scalar step size together with an averaging step, it can be shown \cite{KY03}, that one can achieve the asymptotic convergence rate of a matrix step size.  
This is also explored in the context of  LMS algorithms in \cite{YKI03,YK05a,YK05b}.

\subsection{Consensus Stochastic Approximation Algorithms} \label{sec:consensuslms}
There has  been much \index{diffusion LMS algorithm}
 recent work on {\em diffusion} and {\em consensus} stochastic approximation algorithms, where multiple stochastic approximation algorithms communicate with each other over a graph.
 This area has been pioneered by  Sayed  (see \cite{Say14b} and references therein) and shows remarkable potential in a variety of distributed processing applications.
 
 In this section we briefly analyze the consensus stochastic approximation algorithm.
 The consensus  stochastic approximation algorithm is of the form
\beq  \model_{k+1} = A\, \model_k + \stepsize H(\model_k,\state_k),  \label{eq:diffsa} \eeq
where $A$ is a symmetric positive definite stochastic matrix.
For $A=I$, (\ref{eq:diffsa}) becomes the standard stochastic approximation algorithm. \index{stochastic approximation algorithm! consensus}
\index{consensus! stochastic approximation}

One can analyze the consensus algorithm as follows. Define the matrix $Q$ such that  $A = \exp(Q\epsilon)$ where $\exp(\cdot)$ denotes the matrix exponential.
So $Q$ is proportional to the matrix logarithm of $A$. (Since $A$ is positive definite, the real-valued  matrix logarithm always exists.)  Indeed
$Q$ is a generator matrix with $Q_{ii} < 0$ and $Q \one = 0$. Then  as $\epsilon$ becomes sufficiently small
(recall the ODE analysis applies for  the interpolated process with $\epsilon \rightarrow 0$) since $ \exp(Q\epsilon) \approx I + \epsilon Q + o(\epsilon)$,  one can express (\ref{eq:diffsa}) as
\beq   \model_{k+1} =  \model_k + \stepsize \big(Q \model_k + H(\model_k,\state_k) \big).  \label{eq:consensusode} \eeq
Therefore, the   consensus ODE associated with (\ref{eq:diffsa}) is 
\beq  \frac{d\model}{dt} = Q \model + \E_{\belief_\model} \{ H(\model, \state) \} . \label{eq:c_ode}\eeq
Typically $Q$ is chosen such that at the optimum value $\th^*$, $Q \model^* = 0$ implying that the consensus ODE and original ODE have
identical attractors.

 To summarize, despite at first sight looking different to a standard stochastic approximation algorithm, the consensus stochastic
approximation algorithm (\ref{eq:diffsa}) in fact is equivalent to a standard stochastic approximation algorithm (\ref{eq:consensusode}) by taking the matrix logarithm of the 
consensus matrix $A$. \cite{NKY13,NKY13a} describe such algorithms for tracking the equilibria of time varying games. One can also consider consensus in the space of distributions and tracking such a consensus distribution as described in \cite{KTY09}. The variance of the consensus LMS using the diffusion approximation is analyzed in
\cite{Kri16}.

\clearpage

\bibliographystyle{plain}

\bibliography{$HOME/styles/bib/vkm}

\end{document}